\documentclass[12pt,oneside,titlepage]{amsart}
\usepackage[foot]{amsaddr}
\usepackage[utf8]{inputenc}
\usepackage[english]{babel}
\usepackage{enumitem}
\usepackage{csquotes}
\usepackage{bbm}
\usepackage{subcaption}
\usepackage{float}
\usepackage{fancyhdr}
\usepackage{etoolbox}
\usepackage{hyperref}
\usepackage{amsmath, amssymb}
\usepackage{mathrsfs}
\usepackage{palatino}


\newcommand{\mycomment}[1]{}

\usepackage{graphicx}
\usepackage{wrapfig}
\usepackage{caption}

\usepackage[top=2.5cm, bottom=2.5cm,left=2.5cm,right=2.5cm]{geometry}

\usepackage{amsthm}
\usepackage{theoremref}

\newtheorem{theorem}{Theorem}[section]
\newtheorem{lemma}[theorem]{Lemma}

\newtheorem{prop}[theorem]{Proposition}
\newtheorem*{assumpt}{Assumptions}

\theoremstyle{remark}
\newtheorem{remark}[theorem]{Remark}

\theoremstyle{definition}
\newtheorem{de}[theorem]{Definition}

\usepackage[usenames,svgnames]{xcolor}

\newcommand{\black}{\color{black}}

\newcommand{\er}{Erdős–Rényi }

\DeclareMathOperator{\E}{\mathbb{E}}

\DeclareMathOperator{\PP}{\mathbb{P}}
\DeclareMathOperator{\p}{\mathbb{P}}

\DeclareMathOperator{\quench}{{\textbf{\textup{P}}}}
\DeclareMathOperator{\punifan}{\p^{\textup{an}}}
\DeclareMathOperator{\pmuan}{\p_{\mu}^{\textup{an}}}

\DeclareMathOperator{\N}{\mathbb{N}}

\DeclareMathOperator{\cQ}{\mathcal{Q}}

\DeclareMathOperator{\mumax}{\mu_{\textup{in}}^{\textup{max}}}
\DeclareMathOperator{\muin}{\mu_{\textup{in}}}

\DeclareMathOperator{\bin}{\text{Bin}}
\DeclareMathOperator{\var}{\textup{Var}}
\DeclareMathOperator{\be}{\text{Be}}
\DeclareMathOperator{\pconv}{\xrightarrow[\textit{n} \to +\infty]{\quad \p \quad}}

\DeclareMathOperator{\trex}{\textit{Tx}}

\usepackage{stmaryrd}

\DeclareMathOperator{\mass}{\textup{\textbf{m}}}

\newcommand{\mindeg}{\delta_+}
\newcommand{\maxdeg}{\Delta_{+}}
\newcommand{\event}{{\mathcal{E}^+}}

\DeclareMathOperator{\pat}{\mathfrak{p}}
\newcommand{\one}{\textup{\textbf{1}}}
\newcommand{\entropy}{\textup{H}}
\newcommand{\overh}{\overline{\entropy}}
\newcommand{\tent}{t_{\textup{ent}}}
\newcommand{\tcrit}{t_{\lambda}}
\newcommand{\veps}{V_\eps}

\newcommand{\tmix}{t_\textup{mix}}

\newcommand{\whp}{{w.h.p.~}}
\newcommand{\iid}{{i.i.d.~}}
\newcommand{\rhs}{{r.h.s.~}}
\newcommand{\lhs}{{l.h.s.~}}

\newcommand{\eps}{\varepsilon}

\newcommand{\w}{\boldsymbol{w}}

\usepackage{todonotes}

\begin{document}

\title[Mixing cutoff for SRWs on the Chung--Lu digraph]{Mixing cutoff for simple random walks \\ on the Chung--Lu digraph}
\author{Alessandra Bianchi}
\address{Dipartimento di Matematica "Tullio Levi-Civita",
Universit\`a di Padova, Via Trieste 63,
35121 Padova, Italy.}\email{alessandra.bianchi@unipd.it}
\author{Giacomo Passuello}
\email{giacomo.passuello@phd.unipd.it}

\date{\today}
\numberwithin{equation}{section}

	\begin{abstract}
In this paper, we are interested in the mixing behaviour of simple random walks 
on inhomogeneous directed graphs.
We focus our study on Chung--Lu digraphs, which are inhomogeneous networks that generalize 
 \er digraphs, and where edges are included independently and according to given Bernoulli laws. 
To  guarantee that a unique equilibrium measure exists with high probability, we assume  
that the average degree grows logarithmically in the size $n$ of the graph.
In this weakly sparse regime, we  prove that the total variation distance to equilibrium 
displays a cutoff behaviour at the entropic time of order $\log n / \log\log n$. 
Moreover, we prove that on a precise window the cutoff profile converges to the Gaussian tail function.
This is qualitatively similar to what was proved in \cite{BCS18,BCS19,CCPQ} 
for the directed configuration model. 
In terms of statistical ensembles, our analysis provides an extension of these cutoff results 
from a hard to a soft-constrained model.


\noindent
    {\it Acknowledgments.}
The authors are grateful to Matteo Quattropani for useful discussions and suggestions.
We also acknowledge the anonymous referees for carefully reading the first version of the manuscript 
and making several valuable comments.
The work is partially funded by the University of Padova through the BIRD project 239937 “Stochastic dynamics on graphs and random structures”, and by the INdAM-GNAMPA Project, CUP E53C23001670001.

	\end{abstract}

	\maketitle

\section{Introduction}

The analysis  of stochastic processes evolving on a random structure is a fundamental tool 
for the understanding of many real-world systems defined on deterministic, but huge, networks.
Applications range from physical to biological systems, from computer science to economics and 
social sciences, and the subject has become in the last twenty years a great source of mathematical problems (see, e.g., the survey \cite{Cooper12} and the book \cite{vdH16}).

In this paper, we are interested in the mixing time of simple random walks moving on 
random \emph{digraphs} (i.e., directed graphs). In this setting, the random walk loses the reversibility property, 
and a key problem, while studying the convergence to equilibrium, is to deal with an unknown 
stationary measure, whose characterization represents itself an important theoretical challenge (see, e.g., \cite{chen14, chen17, GvHL18, CQ20}).
Starting from the seminal work of E.~Lubetzky and A.~Sly \cite{ls10}
on the undirected regular random graph,
a series of techniques 
have been devised in order to study the convergence to equilibrium,
and in particular to establish
the occurrence of the so-called \emph{cutoff phenomenon}.  
A Markov chain exhibits a mixing cutoff if there exists a time scale, a function of the size of the system,
on which the distance to equilibrium displays an abrupt decay.
This limit behaviour highlights a \emph{phase transition}, which is visible at rougher time scales.
We refer to \cite{Dia96, LevPer:AMS2017} for an introduction on the topic.

Showing the existence of such a universal profile requires a deep understanding of the interactions between environment and dynamics.
It was observed in \cite{BS17, B} that \emph{non-backtracking} random walks allow a fine control on both the equilibrium measure and the diffusive properties of the system.
With this regard, digraphs constitute a good framework to deal with.
We refer in particular to the works of C.~Bordenave, P.~Caputo and J.~Salez \cite{BCS18, BCS19}, as well as to \cite{CCPQ}, for results concerning the \emph{directed configuration model}, and to \cite{CQ21} for consequent achievements on \emph{PageRank surfers}.
Such \emph{inhomogeneous} models serve as natural tool for studying  dynamics on internet networks. 

In the present work, we analyse the motion of a random walk on a \emph{Chung--Lu digraph}.
This is an inhomogeneous random network obtained 
by sampling edges independently via vertex weights, which represent fixed average degrees.
This setting clearly includes, as particular cases, the directed homogeneous \emph{\er graph} and the 
\emph{stochastic block model}.
To ensure that the random graph is strongly connected, and hence to guarantee the  
uniqueness of the equilibrium measure, we will work on a \emph{weakly sparse regime},
where the average vertex degrees grow as $\log n$, $n$ being the size of the graph.

Our study will mainly refer to the techniques introduced in \cite{BCS18, BCS19}
to deal with the dynamics on the directed configuration model in the sparse regime.
As highlighted in these papers (see also \cite{CQ20, CQ21, CCPQ, ACHQ} for further developments),  
two fundamental statistics for the characterization of the mixing time
are the \emph{in-degree distribution}, which provides an easily computable approximation of the reversible measure,
and the \emph{entropy} of the graph, which measures  the spread of the random walk among the network. 
However, a main hurdle in implementing theses ideas in 
our framework
is that vertex degrees are random, as well as the corresponding in-degree distribution.
To overcome this difficulty we shall introduce an approximated, but deterministic,
 in-degree distribution (see \eqref{muin}), and then leverage on some \emph{concentration results}
 on the vertex degrees  in order to control this approximation error along the dynamics 
 and to characterize asymptotically the entropy (see \eqref{entropy} and Proposition \ref{prop-entropy}).
By implementing this entropic method, devised in \cite{BCS18, BCS19},  we will prove that 
under suitable assumptions the dynamics exhibits a cutoff phenomenon at a time of order 
$\log n / \log\log n$.
Moreover, we will show that, in an appropriate time window, the cutoff profile approaches a 
Gaussian tail function.
This work can be seen as a generalization of the cutoff results
achieved in \cite{BCS18,CCPQ}, where hard constraints on vertex degrees are replaced with a softer randomized version.

\subsection{Graph setting} \label{setting}

	Let $[n]:=\{1,\dots,n\}$ represent a set of vertices of size $n\in\N$, and consider two sequences $(w^-_x)_{x\in[n]}$ and 
$(w^+_x)_{x\in[n]}$ of positive numbers, called weights, such that 
	\begin{equation}
	 \sum_{x \in [n]}w_x^+=\sum_{x \in [n]}w_x^-=: \w(n) =\w \, . 
	 \end{equation} 
	We consider a directed version of the Chung--Lu model, where two distinct vertices $x, y\in[n]$ are connected by an oriented edge from $x$ to $y$, in short $x\to y$, independently and with probability 
	\begin{equation}\label{prob_conn}
	p_{xy}=w^+_x w^-_y\frac{\log n}{n} \wedge 1\,,\quad \forall 
	x,y \in [n]\,, x\neq y\,.
	\end{equation}
We will denote by $\p= \p_n^{w^{\pm}}$ the law of this Chung-Lu random graph 
and by $\E$ the corresponding average, and write $G$ for a given realization of the graph.

		\begin{remark}
		The standard non-oriented Chung--Lu model, introduced in \cite{ChungLu},
		 is defined through a sequence of positive weights $(\widetilde w_x)_{x\in[n]}$ and connection probabilities		
			$$
			p^{\textup{(CL)}}_{xy}:=\frac{\widetilde w_x \widetilde w_y}{\ell_n} \wedge 1, \quad  \forall x\neq y \in [n],  \qquad \text{where} \qquad \ell_n:=\sum_{x \in [n]}\widetilde w_x.
			$$
This can be easily adapted to the  above directed framework taking two sequences $(\widetilde w_x^{\pm})_{x\in[n]}$ with equal sum,  and setting
\begin{equation}\label{prob_DCL}
		p^{\textup{(DCL)}}_{xy}:=\frac{\widetilde w_x^+ \widetilde w_y^-}{\ell_n} \wedge 1, 
		\quad \forall x\neq y \in [n],\qquad \text{where} \qquad \ell_n:=\sum_{x \in [n]}\widetilde w_x^+=\sum_{x \in [n]}\widetilde w_x^-\,.
\end{equation}
Choosing $\widetilde w_x^\pm=w_x^\pm \w\frac{\log n}{n}$ and plugging this value in \eqref{prob_DCL},
 we get that $\ell_n = \sum_{x \in [n]}\widetilde w_x=\w^2\frac{\log n}{n}$,
and we recover our model.
	\end{remark}
\vspace{0.5cm}
	
As main observables on this random structure, we introduce the random out-degree of a vertex $x \in [n]$, 
denoted by $D_x^+$, and set
\begin{equation}
\mindeg:= \min_{x\in[n] }D_x^+ \quad \mbox{ and }\quad \maxdeg:= \max_{x\in[n]} D_x^+\,, 
\end{equation}
which are, respectively,  the minimum and  maximum out-degree of the random graph. 
With obvious notation, we introduce also the corresponding in-degrees random variables 
$(D_x^-)_{x\in[n]}$, $\delta_-$ and $\Delta_-$. 

By assumption, the out- and in-degrees of  each vertex $x \in [n]$  are distributed as a sum of 
independent Bernoulli random variables of parameters $p_{xy}$ and $p_{yx}$ respectively, for $ y \in [n] \setminus \{x\}$.
In particular, their averages  are easily given by 
	\begin{equation}\label{media_grado}
			\E[D_x^\pm]=\sum_{y \in [n]\setminus \{x\}} w_x^\pm w_y^\mp \frac{\log n}{n}. 
		\end{equation}
Along the paper, we will use the following Landau symbols: 
given two real sequences $(a_n)_{n\in\N}$ and $(b_n)_{n\in\N}$ such that 
$\lim_{n\to\infty}\big| \frac{a_n}{b_n}\big|=\ell\,,$ we will write
$a_n=o(b_n)$ (resp. $a_n=\Omega(b_n)$, $a_n=O(b_n)$, $a_n=\Theta(b_n)$, and $a_n\gg b_n$) if 
$\ell=0$ (resp. $\ell>0$, $\ell<\infty$, $0<\ell<\infty$, and $\ell=\infty$).
%
%
In Subsection \ref{sec_ass}, we will set assumptions on $(w^{\pm}_x)_{x\in[n]}$ 
which imply, in the above notation, that $\E[D_x^+]= \Theta(\log n)$, for $x \in [n]$. 
The corresponding random graph will be then in a \emph{weakly sparse regime}.

At last, note that the \er digraph with connection probability $p =\lambda \log n/n$, for \mbox{$\lambda >0$}, corresponds to a homogenous Chung--Lu digraph with constant weights $w^{\pm}_x \equiv \sqrt{\lambda}$.

\subsection{Simple random walk on the Chung--Lu digraph} \label{SRW}
	The Chung--Lu model that we have just portrayed offers a good framework to study random dynamics. 
	We consider the discrete time simple random walk, $(X_t)_{t\in\mathbb{N}}$,  whose transition matrix is
\begin{equation}
	P(x,y):=\begin{cases}
		\frac{1}{D^+_x} & \textup{if } x\to y  \\
				0 &\textup{otherwise}
			\end{cases},
		\,\,\quad \forall x,y \in [n].
\end{equation}
 For every time $t>0$
(for the sake of simplicity $t$ has to be understood as an integer, or its integer part),
 we denote with $P^t(\cdot,\cdot)$ its related $t$-step transition kernel, while
for an oriented path $\pat=(x_0,\dots,x_t)$ in the graph,  we define  the probability mass of $\pat$ as
\begin{equation}\label{mass}
	\mass(\pat):=\prod_{i = 0}^{t-1} P(x_i,x_{i+1})\,,
\end{equation} 
which corresponds to the probability that a random walk starting at $x_0$ follows the trajectory $\pat$.
We point out that $\mass(\cdot)$, as the transition kernel $P(\cdot,\cdot)$, is a random object whose dependence on the random
graph is implicit in the notation. 

For any given realization $G$ of the random graph, 
we will consider the quenched law $\quench^G_\mu$ of the random walk with initial distribution $\mu$,
which is the probability measure acting on the set of trajectories realized on the given graph $G$. 
Averaging over all graph realizations, we obtain the corresponding annealed law $\p^{\textup{an}}_\mu$, which is defined by $\p^{\textup{an}}_\mu(A):=\E[\quench_\mu(A)]$ for every measurable set  $A$ of trajectories of the random walk.
 In our framework the random structure is fixed once forever. We refer to \cite{CQ21tricotomy,AveGulVDHdH:DirConf2018,AGHHN,ST} for the analysis of dynamic networks. 

\subsubsection{\textbf{Uniqueness of the invariant distribution}}

As long as a realization $G$ of the Chung--Lu digraph is strongly connected, i.e. 
there exists a directed path among  every couple of vertices $x,y\in[n]$, the irreducibility condition of  simple random walks is satisfied and this guarantees that  there exists a unique 
invariant measure $\pi$ on $[n]$ such that $\pi P=\pi$.

Then, we are at first interested in finding sufficient conditions which ensure strong connectivity 
with probability tending to 1 as $n\to \infty$ (in short \emph{with high probability} or simply \emph{w.h.p.}).
It was proved in \cite{CF12} that  the \er digraph with parameter $\lambda \log n/n$,
where $\lambda >1$, is \whp strongly connected. 
Provided that there exists a constant $\lambda$ such that $w^+_xw_y^-\ge \lambda>1$ 
for every $x,y \in [n]$, and since the strong connectivity is a monotone increasing property of graphs,
a simple coupling argument leads to the same conclusion for the Chung--Lu graph.
Hence, this condition guarantees the existence and uniqueness of the invariant distribution $\pi$,
and it will be part of the set of assumptions on the graph setting that will be given below, before stating the main results.\\

Provided that the stationary distribution is unique, the main goal of this work is to characterize the mixing time of the random walk, which is defined, for any initial state $x \in [n]$ and any precision $\eps\in(0,1)$,  as
\begin{equation}
	\tmix^{(x)}(\eps):=\inf\{t > 0: \,\|P^t(x,\cdot)-\pi\|_{\textup{TV}} \le \eps \}\,, 	
\end{equation}

where $\|\mu-\nu\|_{\textup{TV}}:=\frac{1}{2}\sum_{x\in[n]}|\mu(x)-\nu(x)|=\sum_{x \in [n]}\left[\mu(x)-\nu(x)\right]^+$, is the total variation distance among the probability measures $\mu$ and $\nu$. Here $[u]^+:=\max\{0,u\}$, for $u\in\mathbb{R}$. 

We stress once more that the mixing time depends on the realization $G$ of the graph, though 
the dependence is implicit in the notation. 
We will prove that our estimates on $\tmix^{(x)}(\eps)$ hold in $\p$-probability as $n\to\infty$.

\subsubsection{\textbf{In-degree distribution}}

One of the main hurdles to estimate the mixing time of simple random walks on digraphs is that 
the stationary measure $\pi$ cannot be explicitly computed. 
In this respect, a useful tool is provided by the following probability measure on the set $[n]$, 
\begin{equation} \label{muin}
	\muin(x):=\frac{w^-_x}{\w}, \quad \text{for } x \in [n]\,.
\end{equation}
The measure $\muin$ can  be seen as an approximate averaged in-degree distribution. 
Specifically, under the upcoming assumption \eqref{eq1}, we can deduce that for large $n$
\begin{equation}
	\E[D_x^-] = \w \frac{\log(n)}{n} w_x^- (1+o(1)), \quad \quad
	\sum_{x \in [n]}\E[D_x^-] = \w^2\frac{\log(n)}{n} (1+o(1)),
\end{equation}
and then, taking the ratio among the two terms, we get
\begin{equation}\label{muin2}
\muin(x)	= \frac{\E[D_x^-]}{\sum_{x\in [n]}\E[D_x^-]}(1+o(1))\,.
\end{equation}
The measure $\muin$  will then simply referred to as in-degree distribution, and it will
naturally appear through the proofs as a fundamental object in understanding the mixing mechanism 
of the dynamics.

\subsection{Assumptions and main results}\label{sec_ass}
\begin{assumpt}
	We assume that:
	\begin{enumerate}
		\item There exist constants $M_0, M_1 >1$ such that, for every $n\in\N$,  
			\begin{equation} \label{eq1}
				\begin{split}
					 M_0 \le  w^+_x \le M_1 < +\infty, \quad \forall x \in [n]; 
				\end{split}
			\end{equation}
		\item There exist constants $M_2>0$ and $0<\eta<1$ such that, for every $n\in\mathbb{N}$,
			\begin{equation} \label{assumption} 
				\begin{split}
					\sum_{x \in [n]}{(w_x^-)}^{2+\eta} \le M_2n.
				\end{split}
			\end{equation}
	\end{enumerate}
\end{assumpt}
\subsubsection*{\textbf{Immediate consequence of the assumptions.}} 
Notice that, as a consequence of \eqref{eq1}, 
$$\w = \Theta(n)\quad \mbox{ and }\quad \E[D_x^{+}]=\Theta(\log n)\,,\,\,\forall x \in [n]\,.$$
Moreover, exploiting \eqref{assumption}, we get that $\max_{x \in [n]} {w_x^-}^{2+\eta}\le M_2 n$,
and	thus
	\begin{equation} \label{maxa-}
		{w_x^-} \le (M_2 n)^{\frac{1}{2+\eta}} \le (M_2 n)^{\frac{1}{2}-\frac\eta6},\quad \forall x \in [n]\,,
	\end{equation}
which in turn implies, by \eqref{prob_conn},  that
	\begin{equation} \label{maxp_xy}
		p_{\text{max}}:=\max_{x\neq y \in [n]} p_{xy} = o(n^{-\frac 1 2-\frac \eta 7}),
	\end{equation}
and 
\begin{equation} \label{mumax}
	\mumax:=\max_{x \in [n]}\muin(x) = O(n^{-\frac{1}{2}-\frac\eta6}).
\end{equation}
In particular, following the terminology introduced in \cite{CQ21}, the above assumptions imply that $\muin$ is a \emph{widespread measure}.

\subsubsection{\textbf{Main results}}	
 Before stating the main results, and following the  procedure traced in \cite{BCS18}, 
 we need to introduce two fundamental quantities that will characterize the mixing time and the cutoff window of the dynamics. 
  We define the entropy $\entropy$ of the Chung--Lu model as the mean logarithmic out-degree of a vertex sampled 
  from $\muin$ (see \eqref{muin}). 
 Formally, denoting by $V$ a random vertex in $[n]$ with law $\muin$, we set  
		\begin{equation} \label{entropy}
			\entropy:= \E \times \muin\left[\log \left(D_V^+\vee 1\right)\right] ,
		\end{equation} 
and let $\sigma^2$ be the corresponding variance, hence given by
	\begin{equation} \label{variance-sigma}
		\sigma^2:= \E \times \muin\left[\log^2 \left(D_V^+\vee 1\right) \right]-\entropy^2\,.
	\end{equation}
We also define the \emph{entropic time}
	\begin{equation}\label{t_ent}
		\tent:=\frac{\log n}{\entropy}\,,
	\end{equation}
which we will show to be precisely the mixing time of the dynamics. 	
In this sense, it is useful to state the following preliminary result
which provides the asymptotic behaviour of $\entropy$ and $\sigma^2$, as $n\to\infty$. 
\begin{prop} \label{prop-entropy} 
 Under the assumptions \eqref{eq1} and \eqref{assumption}, it holds
			\begin{equation}
			{\entropy}={\log\log n} (1+o(1)), \quad \quad
			\sigma^2 = O(\log\log n).
		\end{equation}
	\end{prop}
While the proof of the above proposition is postponed to Subsection \ref{sec_entropy},
we can immediately argue  that the entropic time $\tent$ is asymptotically 
of order $\log n/\log\log n$.
With that in mind, and with the usual convention that the discrete dynamics is evaluated in the integer part of each considered time, we can state our main results.
\begin{theorem}[Cutoff] \label{thm_main}
	Let $\beta > 0$. It holds 
	\begin{equation}\label{main1}
		\min_{x \in [n]}\|{P}^{\, (1 - \beta) \, \tent}(x,\cdot)-\pi\|_{\textup{TV}}\pconv 1.
	\end{equation}
	and
	\begin{equation}\label{main2}
		\max_{x \in [n]}\|{P}^{\, (1 + \beta) \, \tent}(x,\cdot)-\pi\|_{\textup{TV}}\pconv 0.
	\end{equation}
\end{theorem}

\begin{remark}
	By the monotonicity properties of the function 
	$t \mapsto \|P^t(x,\cdot)-\pi\|_{\textup{TV}}$ for $x \in [n]$, 
	we get that \eqref{main1} holds for any $t \le (1-\beta)\tent$, 
	while \eqref{main2} holds for any $t \ge (1+\beta)\tent$.
\end{remark}

The statement can be rephrased as follows: for every precision $\eps \in (0,1)$,	
\begin{equation}
	\max_{x \in [n]}\Bigg|\frac{\tmix^{(x)}(\eps)}{\tent}-1\Bigg|\pconv 0.
\end{equation}
This means that regardless of the starting point and the precision, the random walk takes, 
with high probability for $n$ large enough,  $\log n/\log \log n$ steps to mix.

This abrupt transition from $1$ to $0$ of the distance to stationarity can be further explored by zooming in around 
the cutoff point $\tent$, and in particular by taking an appropriate window of size $\textbf{\textup{w}}_n$, with
\begin{equation}\label{window}
\textbf{\textup{w}}_n:=\frac{\sigma}{\entropy}\sqrt{\tent} \,.
\end{equation}
To avoid pathological situations, we will assume that $\sigma^2$ is non-degenerate 
in the following weak sense: there exists $\delta>0$ such that
\begin{equation} \label{non-degenerate-variance}
	\sigma^2 \gg \frac{(\log \log n)^{2+\frac{\delta}{\delta+2}}}{(\log n)^{\frac{\delta}{\delta+2}}}.
\end{equation}
Note that as $\delta\to\infty$, the r.h.s. reaches the order $(\log\log n)^3/\log n$, providing a non-degeneracy condition  similar to that given in \cite{BCS18}.

The next result shows that, inside this window and under this assumption, the cutoff shape approaches 
the tail distribution of the standard normal.

\begin{theorem}[Cutoff window] \label{thm_window}
Assume that the variance $\sigma^2$ satisfies the non-degeneracy condition \eqref{non-degenerate-variance}.
Then, for  $\tcrit:=\tent+\lambda \textbf{\textup{w}}_n+o(\textbf{\textup{w}}_n)$ with 
$\lambda \in \mathbb{R}$ fixed, it holds 

	\begin{equation}
		\max_{x \in [n]}\left|\|{P}^{t_\lambda}(x,\cdot)-\pi\|_{\textup{TV}}-\frac{1}{\sqrt{2\pi}}\int_{\lambda}^{+\infty}e^{-\frac{u^2}{2}}\,du\right|\pconv 0.
	\end{equation} 
\end{theorem}

\begin{remark}
Notice that the statements of Theorems~\ref{thm_main} and \ref{thm_window} can be easily extended to Chung--Lu digraphs 
with random sequences of weights $(W^+_1,\dots,W^+_n)$ and  $(W^-_1,\dots,W^-_n)$ which satisfy \emph{a.s.} the constraints \eqref{eq1} and \eqref{assumption}.\end{remark}

\section{Proof Outline and main ingredients} \label{outline}

\subsection{{General strategy}}

A main hurdle in the analysis of the mixing time of simple random walks on  digraphs is the lack of an explicit formula
for the stationary measure $\pi$. To cope with that, we will  introduce an explicit probability measure $\widetilde{\pi}$ 
that well approximates $\pi$ itself. 

Using this idea, and looking first at an upper bound on the mixing time,
by the triangle inequality we can write
	\begin{equation} \label{eq3}
		\|P^{t}(x,\cdot)-\pi\|_{\textup{TV}} 
		\le \|P^{t}(x,\cdot)-\widetilde{\pi}\|_{\textup{TV}}+\|\widetilde{\pi}-\pi\|_{\textup{TV}},\quad  \forall x \in [n]\,.
	\end{equation}
Note that if the first term in the \rhs is $o_{\p}(1)$ uniformly in $x \in [n]$, 
then the same must hold for the second term since
\begin{equation} \label{triangle-ineq}
\|\widetilde{\pi}-\pi\|_{\textup{TV}}= \|\widetilde{\pi}-\pi P^{t}\|_{\textup{TV}} 
=\sum_{x\in[n]}\pi(x)\|P^{t}(x,\cdot)-\widetilde{\pi}\|_{\textup{TV}}\,. 
\end{equation}
 
This is what we will prove whenever $t\ge(1+\beta)\tent$, taking  $\widetilde{\pi}:=\muin P^{h_\eps}$, with $\eps>0$ and
\begin{equation} \label{h_eps}
	h_\eps:= \frac{\eps \log n}{20 \entropy}\,.
\end{equation}
More precisely, we will prove the following slightly weaker condition 
\begin{equation}\label{meta}
\|P^{t}(x,\cdot)-\widetilde{\pi}\|_{\textup{TV}}=o_{\p}(1) \,,\quad \forall x\in\veps \,,
\end{equation}
where $\veps$ is a subset of $[n]$  whose vertices have a locally tree-like out-neighbourhood. 
This result will be sufficient to derive a proper upper bound on the mixing time as stated in \eqref{main2} of Theorem \ref{thm_main}.

As a further main tool to obtain \eqref{meta}, which also enters in the proof of 
the lower bound on the mixing time, 
we will introduce a suitable set of $t$-length paths, called \emph{nice paths}, 
that will be shown to be typical trajectories of the simple random walk.
Taking advantage of their properties, we will prove that, for any $\delta >0$,
	\begin{equation} 
	\begin{split}
		\|P^{t}(x,\cdot)-\widetilde{\pi}\|_{\textup{TV}}\le
		 \cQ_{x,t}\Big(\frac{1}{n\log^3n}\Big)+3\delta \,,
	\end{split}
\end{equation}
where for $x\in [n]$ and $\theta \in(0,1)$, $\cQ_{x,t}(\theta)$ is the quenched probability
that the mass of a path of length $t$ selected by a random walk with initial point $x$ is bigger than $\theta$.
Formally
	\begin{equation} \label{quenched-theta} 
		\cQ_{x,t}(\theta):=\quench_x(\mass(X_0,X_1,\dots,X_t)>\theta),
	\end{equation}
	where $\quench_x$ is the quenched law of a random walk starting at $x$ as in Section \ref{SRW},
	 and $\mass(\cdot)$ is the mass of a path as given in \eqref{mass}.

A similar approach can be implemented to obtain a lower bound on the total variation distance
as stated in \eqref{main1}. 
In particular,  for  $t=(1-\beta)\log n$ and $\theta= \log^{a}n/n$ (with a suitable $a\in\mathbb{N}$), 
it will lead to the inequality 
	\begin{equation} \label{lower-bound}
			\cQ_{x,t}(\theta)\le \|\widetilde{\pi}-P^t(x,\cdot)\|_{\textup{TV}}+ o_{\p}(1)\,.
	\end{equation}
	The function $\cQ_{x,t}(\theta)$  is thus one of the main characters of our analysis, 
	and it will carry a very powerful limit result: 
	in Theorem~\ref{thm_LLN} we will observe that according to the choices  $(t,\theta)$, 
	it may vanish or saturate to $1$. 
	This dichotomy will actually conclude the proof of the cutoff regime (Theorem \ref{thm_main}), and provide the main strategy for the proof of cutoff profile (Theorem \ref{thm_window}).

We would like to emphasize again that the overall strategy of our proofs follows the entropic method developed by Bordenave, Caputo, and Salez in \cite{BCS18, BCS19} for analyzing random walks on sparse directed configuration models. While we draw on these ideas, our implementation occurs in a quite different setting, necessitating significant modifications.
In the case of directed configuration models, the analysis often relies on combinatorial computations, which are feasible due to the deterministic nature of in- and out-degrees. However, this approach is not applicable to the Chung–Lu setting, where the in- and out-degrees are themselves random. Instead, the Chung–Lu model benefits from the independence of edges, a property we crucially exploit in our analysis, along with appropriate concentration inequalities for the in- and out-degrees.

Finally,  note that our results are consistent with Theorem 3 in \cite{ls10}, which is set in the context of undirected regular random graphs in a weakly sparse setting. Although we are not aware of analogous results for the undirected Chung–Lu model in this regime, we believe that similar conclusions can be drawn for a broad class of undirected random graphs. 
However, in this setting, the speed of the random walk enters the game and needs to be properly
analyzed (see e.g. \cite{BLPS} for the study of sparse undirected graphs).   
It is also worth mentioning the analysis in \cite{FR}, where the authors characterize the mixing time of random walks on Erdős–Rényi graphs with an average degree up to the order of $\sqrt{\log n}$, which is slightly below the assumptions of the present study.

\subsection{{Typical paths and tree-like neighbourhoods}}\label{sec_trees}
We explain here the properties that a path of length $t$ has to satisfy in order to be called \emph{nice}.
\begin{de}[Nice path] \label{def}
	Let $\gamma=\frac{\eps}{80}$, $\eps \in (0,1)$, $h_\eps$ as in \eqref{h_eps}, and
	\begin{equation}
		s:=(1-\gamma)\tent, \quad \quad t:=s+h_\eps+1=(1+3\gamma)\tent+1.
	\end{equation}
We say that a path $\pat=(x,x_1,x_2,\dots,x_{t-1},y)$ of length $t$ from $x$ to $y$ is \textit{nice} if
	\begin{enumerate} 
		\item[(i)] the entire path is such that $\mass(\pat)\le \frac{1}{n \log^3 n}$;
		\item[(ii)] the first $s$ steps are contained in certain tree $\mathcal{T}_x(s)$, defined below;
		\item[(iii)] the last $h_\eps$ steps form the unique path in $G$ of length at most $h_\eps$ to $y$; 
		\item[(iv)] it holds $P(x_s,x_{s+1})=1/D^+_{x_s}\ge \frac 1{C\log n}$, for some constant $C>0$.
	\end{enumerate}
\end{de}
\begin{remark}\label{rem-gen-tree}
Definition \ref{def} and the consequent machinery can be extended to times $t=t_\lambda$, lying in the critical window of Theorem \ref{thm_window}. 
In that case we set $s=t_\lambda- h_\eps$.
\end{remark}

To formalize the above properties, it remains to define the tree $\mathcal{T}_x(s)$.

\subsubsection{\textbf{Construction of the tree $\mathcal{T}_x(s)$}} \label{tree-constr}  
For a given realization of the graph $G$, a fixed root node $x \in [n]$ and a time $s \in \N$, 
we construct with an iterative procedure two sequences $(\mathcal{G}^\ell)_{\ell \ge0}$ and $(\mathcal{T}^\ell)_{\ell\ge0}$ such that, for every $\ell\ge 0$, $\mathcal{G}^\ell$ is a subgraph of $G$ 
with $\ell$ edges, while $\mathcal{T}^\ell$ is a spanning tree of $\mathcal{G}^\ell$. 
The criterion adopted is similar to the one in \cite[Sect.~3.2]{CCPQ}. \\

Set $\overh:=(1+\gamma)\entropy$, where $\gamma=\frac{\eps}{80}$ as in Definition \ref{def}. 
To initialize,  let $\mathcal{G}^0=\mathcal{T}^0:=\{x\}$. 
Then, for $\ell \ge 1$:
\begin{enumerate}
\item Let $\mathcal{E}^{\ell}$ be the set of edges with tails belonging 
	to $\mathcal{G}^{\ell-1}$, and  which have not been yet visited by the first $\ell -1$ 
	iterations of the algorithm. For an edge $e\in\mathcal{E}^{\ell}$,
	 define the cumulative mass
\begin{equation}\label{mass-hat}
\hat{\mass}(e):=\mass(\pat_{x,v_e^-})\frac{1}{D^+_{v_e^-}}\,,
\end{equation}%
	where $v_e^-$ is the tail of $e$, and $\pat_{x,v_e^-}$ denotes the unique path 
	in $\mathcal{T}^{\ell-1}$ from $x$ to $v_e^-$. 
	In particular, $\hat{\mass}(e)$ corresponds to the probability that the random walk follows 
	$\pat_{x,v_e^-}$ and then the edge $e$.
\item Choose $e_\ell\in \mathcal{E}^{\ell}$ such that:
	\begin{enumerate}
		\item[(a)] $v_{e_{\ell}}^-$ is at distance at most $s - 1$ from the root $x$,
		\item[(b)] $\hat{\mass}(e_{\ell})=\max_{e\in \mathcal{E}^{\ell}}\hat{\mass}(e) 
		\text{\quad and \quad}\hat{\mass}(e_{\ell})\ge e^{-\overh s}$.
	\end{enumerate} 
	 If such edge does not exist, stop the procedure and set $\kappa_x\equiv\kappa_x(s) = \ell-1$ ; 
	\item Generate $\mathcal{G}^\ell$ by adding $e_{\ell}$ to 
	$\mathcal{G}^{\ell-1}$ ;
	\item If step (2) does not break the tree structure of $\mathcal{T}^{\ell-1}$, 
	generate $\mathcal{T}^\ell$ 	by adding $e_{\ell}$ to  $\mathcal{T}^{\ell-1}$ and otherwise set 
	$\mathcal{T}^{\ell}=\mathcal{T}^{\ell-1}$.
\end{enumerate}
Note that $\kappa_x\equiv\kappa_x(s)$ is the last step of the iteration, and that it is finite
as the graph itself is finite.
We then set 
$\mathcal{G}_x(s):=\mathcal{G}^{\kappa_x}$ and 
$\mathcal{T}_x(s):=\mathcal{T}^{\kappa_x}$. We observe that $\mathcal{G}_x(s)$ is generated by all paths with mass at least $e^{-\overh s}$ and length at most $s$. \\ 

We will show that the properties of \emph{nice paths} are satisfied \whp for $s$ as in Definition \ref{def} and
 uniformly in all starting points $x\in\veps$,
where $\veps\subset[n]$ is the random set of vertices mentioned in Eq.~\eqref{meta} and defined as follows. 

For $h \in \mathbb{N}$ and $x \in [n]$, let us denote with  $\mathcal{B}^+_x(h)$ (resp.~$\mathcal{B}^-_x(h)$), 
the set of vertices $y \in [n]$ that are connected to $x$ by an oriented path of length at most $h$ 
and starting (resp.~ending) at point $x$. 
They will be called out- (resp.~in-)neighbourhood of $x$ of radius $h$.
Then we set
\begin{equation}  \label{veps}
	\veps :=\{ x \in [n] : \mathcal{B}^+_x(h_\eps) \text{ is a directed tree}\}.
\end{equation}
 As in \cite{ls10}, vertices $x \in \veps$ are named $h_\eps$-roots. 
 We will prove that $V_\eps$ is attractive in a sense that will be specified in Lemma~\ref{lemmaveps}.

\section{Tools}

\subsection{Annealed random walk}\label{annealed random walk}
 

In this section we will give an alternative construction of the annealed law of a random walk.
We will actually generalize this object to the joint annealed law of $K$ independent random walks defined on the same random graph.
This will be used in the forthcoming sections to compute the $K$-th moment of certain quenched statistics.

Let $K \in \mathbb{N}$. Given an initial distribution $\mu$ and a time $T$, we define iteratively the non-Markovian process  
$(X^{(k)})_{k \in \{1,\dots,K\}}$, where $X^{(k)}=(X_t^{(k)})_{0\le t\le T}$ 
is a random walk of length $T$ whose evolution is, for every $k \ge 2$, conditioned to the previous $k-1$ walks.
Formally, every random walk $X^{(k)}$ is defined by the following procedure:
\begin{enumerate}
	\item[(1)] Set $X_0^{(k)} \sim \mu$;
\end{enumerate}
Then for all $t \in \{1,\dots,T\}$:
\begin{enumerate}
	\item[(2)] 
	\begin{enumerate}
		\item[$\bullet$] If $X_{t-1}^{(k)}$ was never visited before by the previous walks or for $s\le t-1$, 
		generate its out-neighbourhood ${B}_{X_{t-1}^{(k)}}:=\mathcal{B}^+_{X_{t-1}^{(k)}}(1)$, 
		according to the probability $\p$, and select a vertex $v$ uniformly at random on ${B}_{X_{t-1}^{(k)}}$ ;
		\item[$\bullet$] If $X_{t-1}^{(k)}$ has been already visited, extract $v$ uniformly at random 
		from the previously generated out-neighbourhood of ${X_{t-1}^{(k)}}$;
	\end{enumerate}
	\item[(3)]  Set $X_t^{(k)}=v$.
\end{enumerate}
The key point of the above construction is that the law of $(X^{(k)})_{k \in \{1,\dots,K\}}$ corresponds 
to the annealed joint law $\p_{\mu}^{\textup{an},K}$ of a system of $K$ independent random walks $(X^{G,k})_{k \in \{1,\dots,K\}}$
(we recall that $G$ denotes a realization of the graph).
Indeed, given a measurable set of trajectories $A\subseteq [n]^{T\times K}$, we have
\begin{align*}
	\p_{\mu}^{\textup{an},K}&((X^{G,1},\dots,X^{G,K})\in A)=\E\big[\quench^{K}_\mu((X^{G,1},\dots,X^{G,K})\in A)\big]=\\ \nonumber =&\E\bigg[\sum_{\{x^{j}_t\}_{t,j}\in [n]^{T\times K}\cap A} \prod_{j=1}^K \mu(x^{j}_0) \prod_{t=0}^{T-1} P(x^{j}_t,x^{j}_{t+1})\bigg]=\\ \nonumber
	=&\E\bigg[\sum_{\{x^{j}_t\}_{t,j}\in [n]^{T\times K}\cap A} \prod_{j=1}^K \mu(x^{j}_0) \prod_{t=0}^{T-1} \E\left[P(x^{j}_t,x^{j}_{t+1})\big| ({B}_{x_r^j})_{r<t} \,,\, ({B}_{x_r^\ell})_{r\le T, \ell<j}\right]\bigg]\,,
	\end{align*}
which characterizes the law of $(X^{(k)})_{k \in \{1,\dots,K\}}$. 

\begin{remark}
Notice that the annealed random walk has an applied interest: its defining algorithm constructs samples of independent random walks
moving on a common structure. Understanding their self-repetition properties could provide information on the 
geometry of the graph, which is very important for statistical inference purposes.
\end{remark}
For a single random walk $X=(X_t)_{t\in\mathbb{N}}$ and any time $s\in\N$, 
we introduce the event that the vertex $X_s$ was never visited before the step $s$, formally written as
\begin{equation}\label{event_L}
\mathcal{L}_{s}=\{X_s\neq X_u ,\, \forall  u \in\{0,\ldots, s-1\}\}\,,
\end{equation}
where for $s=0$, the event $\mathcal{L}_{0}$ should be understood as the whole sample space. 
Using this notation, we are going to prove a result which highlights the role of the measure $\muin$, defined in \eqref{muin}, along the dynamics. Before giving the statement, we recall that $\mumax=\max_{x \in[n]}\muin(x)=O(n^{-\frac12-\frac\eta6})$, as observed in \eqref{mumax}.
\begin{lemma} \label{lemmaiid}
For every initial distribution $\mu$ and any positive $s=O\left(n^{1/2}\right)$, it holds
	\begin{equation}
		\pmuan(X_s=z , \mathcal{L}_{s-1})=\muin(z)\Big[1+O\Big(\frac{1}{\sqrt[3]{\log n}}\Big)\Big].
	\end{equation}
\end{lemma}
\begin{proof}
If $s>1$, and setting $z=z_s \in [n]$, we can write 
	\begin{equation} \label{prima}
		\begin{split}
			\pmuan(X_s=z , \mathcal{L}_{s-1})&=
			\sum_{\substack{z_0, \dots, z_{s-1} \in [n]\\{z_{s-1} \notin \{z_0,\dots,z_{s-2}\}}} } \mu(z_0)
			\E\left[\prod_{i=0}^{s-1}\frac{\one_{\{z_i \to z_{i+1}\}}}{D^+_{z_i}}\right]
			\\&
			= 
			\sum_{\substack{z_0, \dots, z_{s-1} \in [n]\\{z_{s-1} \notin \{z_0,\dots,z_{s-2}\}}} } \mu(z_0)
			\E\left[\prod_{i=0}^{s-2}\frac{\one_{\{z_i \to z_{i+1}\}}}{D^+_{z_i}}\right]\E\left[\frac{\one_{\{z_{s-1} \to z\}}}{D^+_{z_{s-1}}}\right],
		\end{split}
	\end{equation}
Where we used that $\one_{\{z_{s-1}\to z_s\}}$ is independent of the other indicator functions, by definition of $\mathcal{L}_{s-1}$.
From the concentration results on the out-degree $D_x^+$ that will be shown in Subsection \ref{sec_entropy}, 
the conditional average appearing in the last display is given, up to lower order terms, by 
\mbox{$(\E[D^+_{z_{s-1}}])^{-1}=(\w w_{z_{s-1}}^{+}\log n/n)^{-1}(1+O(1/\sqrt[3]{\log n}))$} (see Remark \ref{inverse-remark}). 
Inserting  this value in \eqref{prima}, using that $p_{z_i z_{i+1}}=w_{z_i}^+w_{z_{i+1}}^- \log n/n$, 
and from the explicit form of $\muin$, we get
	 \begin{equation}\label{uff}
	 	\begin{split}
			 \sum_{\substack{z_0, \dots, z_{s-1} \in [n]\\{z_{s-1} \notin \{z_0,\dots,z_{s-2}\}}} } \!\!\!\!\!\!\!\!\!\!\!\!
			 &\mu(z_0)\E\!\left[\prod_{i=0}^{s-2}\frac{\one_{\{z_i \to z_{i+1}\}}}{D^+_{z_i}}\right]\!\muin(z)\Big[1+O\Big(\frac{1}{\sqrt[3]{\log n}}\Big)\Big] \!\!=
 			 \pmuan(\mathcal{L}_{s-1})\muin(z)\Big[1+O\Big(\frac{1}{\sqrt[3]{\log n}}\Big)\Big].
			\end{split}
	 \end{equation}
We now observe that, for every $i \le s-1$, thanks to \eqref{maxp_xy} and our hypothesis on $s$, 
	\begin{equation}
	 	1-o(n^{-\frac \eta 7})=1-s p_{\text{max}} \le \pmuan(\mathcal{L}_{i}) \le 1.
	\end{equation}
Then the claimed statement holds for all $s>1$.

If $s=1$,  being $\mathcal{L}_{0}$ the whole sample space, we get more directly, by the same estimates, 
\begin{equation} \label{s=1}
	\pmuan(X_1=z )=
			\sum_{z_0\in [n]}  \mu(z_0)
			\E\left[\frac{\one_{\{z_0 \to z\}}}{D^+_{z_0}}\right]
		= \muin(z)\Big[1+O\Big(\frac{1}{\sqrt[3]{\log n}}\Big)\Big]\,. \qedhere
	\end{equation}
\end{proof}
\begin{remark}
	By Lemma \ref{lemmaiid} and Remark \ref{inverse-remark}, $\pmuan(X_s=z,\mathcal{L}_{s-1})=\pmuan(\mathcal{L}_{s-1})\muin(z)(1+\epsilon_z),$ where $0<\epsilon_z = O(1/\sqrt[3]{\log n})$ and $\pmuan(\mathcal{L}_{s-1}^c)=o(n^{-\frac \eta 7})$.
	As a consequence
	\begin{equation}
		\begin{split}1 =& \sum_{z \in [n]}\pmuan(X_s=z , \mathcal{L}_{s-1}) + \pmuan(\mathcal{L}_{s-1}^c)	\\
			=&\pmuan(\mathcal{L}_{s-1})\big(1+\sum_{z \in [n]}\muin(z)\epsilon_z\big)+ \pmuan(\mathcal{L}_{s-1}^c)
			= 1+\sum_{z \in [n]}\muin(z)\epsilon_z+ o(n^{-\frac \eta 7})\ ,
		\end{split}
	\end{equation}
	which leads to $\sum_{z \in [n]}\muin(z)\epsilon_z=o(n^{-\frac \eta 7})$. Then, we conclude,
	\begin{equation}\label{tv-cost}
		\begin{split}
			2 \|\pmuan(X_s=\cdot)-\muin\|_{\textup{TV}}
			&\le \sum_{z \in [n]}|\pmuan(X_s=z,\mathcal{L}_{s-1})-\muin(z)|+\pmuan(\mathcal{L}_{s-1}^c)\\
			&\le \sum_{z \in [n]}\muin(z)|\epsilon_z- o(n^{-\frac \eta 7})|+\pmuan(\mathcal{L}_{s-1}^c) \\
			& \le  \sum_{z \in [n]}\muin(z)(|\epsilon_z|+|o(n^{-\frac \eta 7})|)+\pmuan(\mathcal{L}_{s-1}^c) 
			=o(n^{-\frac \eta 7}).
		\end{split}
	\end{equation}
\end{remark}

Let us now define, for every $0<s<t$ the event $\mathcal{A}_{s,t}$ that the trajectory $(X_u)_{s\le u<t}$ has no self-intersections,
formally given by
\begin{equation}\label{event_A}
	\mathcal{A}_{s,t}\equiv \mathcal{A}_{s,t}^X:=\{X_{u}\neq X_v ,\, \forall  u\neq v \in\{s,\ldots, t-1\}\}\,.
\end{equation}
We set also $\mathcal{A}_t:=\mathcal{A}_{0,t}$. 

The next result shows that, if the initial measure $\mu$ is $\textup{Unif}([n])$,  then  the event $\mathcal{A}_T$ is indeed typical for a time $T=\log^2 n$, which is asymptotically much bigger than $\tent$. 
This will be crucial to prove the convergence result inside the cutoff window. 

\begin{lemma} \label{lemma-selfintersection}
Let $T:=\log^{2}n$.  If $\mu = \textup{Unif}([n])$, then there exists a constant $C_1>0$ such that
	$$\pmuan(\mathcal{A}_T^c)\le C_1\log^{4}n/n\,.$$
\end{lemma}
\begin{proof}
Let $\tau$ be the first self-intersection time of $X$, given by
$$ \tau:=\min\{s>0\,:\,\exists u<s \mbox{ such that } X_s=X_u  \}\,,$$	
	 and write
	\begin{equation}\label{tau_split}
		\pmuan(\mathcal{A}_T^c)=\pmuan(\tau < T )= \sum_{t=1}^{T-1} \pmuan(\tau = t),
	\end{equation}
	where	
	\begin{equation} \label{3.12}
		\pmuan(\tau = t)= \sum_{z \in [n]} \Big( \pmuan(X_0=X_t=z,\tau=t)+\sum_{0<s<t}\pmuan(X_s=X_t=z,\tau=t)\Big).
	\end{equation}
	We estimate separately the two terms appearing in the above summation.
	
	The first term can be written as 
	\begin{equation}\label{ti0}
	\begin{split}
		\pmuan(X_0=X_t=z,\tau=t) &= 
		\pmuan(X_t=z,\tau=t|X_0=z) \cdot \pmuan(X_0=z) \\
		& = \p^{\textup{an}}_z (X_t=z,\tau=t)\mu(z) \le  \p^{\textup{an}}_z (X_t=z,\mathcal{L}_{t-1})\mu(z)\\ 
		&= \frac{1}{n}\muin(z) (1+o(1)),
	\end{split}
	\end{equation} 
where the last identity follows from Lemma \ref{lemmaiid} and using that  $\mu = \textup{Unif}([n])$.
Inserting this value in \eqref{3.12} and summing over $z$, we conclude that this term provides an overall 
contribution to $\pmuan(\tau = t)$ equal to $1/n +o(1/n)$. 
	
	Let us turn to the second term.
	For all $s< t\le T$, we introduce the event
	\begin{equation}\label{event_B}
\mathcal{B}_{s,t}\equiv \mathcal{B}_{s,t}^X:=\{X_{v}\neq X_u ,\, \forall  u\in\{0,\ldots, s-1\} 
\mbox{ and }v\in\{s,\ldots, t-1\}\}\,,
\end{equation}
corresponding to the event that the trajectory $(X_v)_{v\in[s,t)}$ does not intersect the trajectory
$(X_u)_{v\in[0,s)}$.
Note that, in this notation, $\mathcal{A}_t = \mathcal{A}_s\ \cap\mathcal{A}_{s,t}\cap \mathcal{B}_{s,t}$, 
and we can write
	\begin{equation} \label{eq3.15}
		\begin{split}
			 \pmuan(X_s=X_t=z,\tau=t) & \le \pmuan(X_s=X_t=z,\mathcal{A}_{t})\\
			 &= 
			\pmuan(X_s=X_t=z,\mathcal{A}_{s} \cap \mathcal{A}_{s,t}\cap \mathcal{B}_{s,t})  \\  
			& =
			\sum_{\substack{
					v \in ([n]\setminus {z})^{s}\\\text{self-avoiding} }} \pmuan(X_t=z,\mathcal{A}_{s,t}|
					(X_k)_{0\le k\le s}=	(v,z),\, \mathcal{B}_{s,t} )\\
					&\hspace{3cm}
				\times \pmuan((X_k)_{0\le k\le s  }=(v,z) ,\, \mathcal{B}_{s,t} ).
		\end{split} 
		\end{equation}%
Thanks to the conditioning, the first factor can be written as 
\mbox{$\widetilde{\p}^{\textup{an}}_z(X_{t-s}=z,\mathcal{A}_{t-s})$} where 
$\widetilde{\p}^{\textup{an}}_z(\cdot)=\widetilde{\E}[\quench_z(\cdot)]$ denotes the annealed law induced 
by a Chung--Lu probability measure $\widetilde{\p}$ on a graph with $n - s$ nodes. 
To the sake of readability we do not stress the dependence of $\widetilde{\p}$ on the vector $v \in ([n]\setminus {z})^s$. 
We conclude observing  that, thanks to Lemma \ref{lemmaiid},
   	\begin{equation}
			\begin{split}
				\widetilde{\p}^{\textup{an}}_z(X_{t-s}=z,\mathcal{A}_{t-s}) \le
				\widetilde{\p}^{\textup{an}}_z(X_{t-s}=z,\mathcal{L}_{t-s}) = \muin(z)(1+o(1)) 
			\end{split} 
	\end{equation}
Plugging this identity in \eqref{eq3.15}, summing over $v \in ([n]\setminus {z})^s$,
and applying once more Lemma \ref{lemmaiid}, we end up with 
	\begin{equation}\label{pezzo2}
		\begin{split}
		\pmuan(X_s=X_t=z,\tau=t) 
		&\le \muin(z) \pmuan(X_s =z,\, \mathcal{A}_{s}\cap \mathcal{B}_{s,t})\\
		& \le \muin(z) \pmuan(X_s =z,\, \mathcal{A}_{s}) \le \muin(z) \pmuan(X_s =z,\, \mathcal{L}_{s-1})\\
		&= \muin(z)^2(1+o(1))
		\end{split}
    \end{equation} 
Inserting this value in \eqref{3.12}, summing over $s<t$ and $z\in[n]$,
and noting that, by assumption \eqref{assumption},  there exists a finite constant $C_1$ such that
\begin{equation}\label{square}
\sum_{z \in [n]} \muin(z)^2 \le M_2n/\w^2\le \frac{C_1}{n}\,,
\end{equation}
we conclude that the contribution  to $\pmuan(\tau = t)$ of this second term is at most $C_1 \frac{T-1 }{n}$. 
The claimed statement follows including these estimates in \eqref{tau_split}.
 \end{proof}

 \begin{remark}
 Note that the bound of order $\log^{4}n/n$ is due to the specific choice of the time $T$.
 The result can be generally stated for any time $T$ which grows poly-logarithmically in $n$, 
 providing an estimate of order $O(T^2/n)$.  
 The requirement over the initial measure can be similarly weakened by replacing $\textup{Unif}([n])$ 
 with a measure $\mu$ sufficiently widespread over $[n]$, so that $\max_{x\in[n]}\mu(x)=O(T/n)$
 and the term in \eqref{ti0} can be properly controlled.  
 \end{remark}

\subsection{Properties of the random graph}

In this section we consider some non-trivial properties of the environment which are the ground floor 
to understand the typical behaviour of random walk paths. 
We will state two main results about the in- and out-neighbourhood of a given vertex, and provide the proof 
of Proposition \ref{prop-entropy} regarding the entropy asymptotics.

\subsubsection{\textbf{Concentration of out-degrees and entropy}}\label{sec_entropy}
Our first two results concern with the out-degree properties of the graph. 
They are straightforward consequences of the Chernoff bounds, which we provide below for the reader's convenience (see Prop. 2.21, \cite{vdH16}).\\

Let $X_i \sim \be(p_i)$, $i = 1,\dots,n$, be independent Bernoulli random variables 
of parameter $p_i \in (0,1)$ and let $X=\sum_{i =1}^n X_i$. Then, for every choice of $t>0$,
\begin{equation} \label{binomial}
\begin{split}
&\PP(X\ge \E[X]+t) \le \exp\left(-\frac{t^{2}}{2(\E[X]+ t/3)}\right)\,,\\
&
\PP(X\le \E[X]-t) \le \exp\left(-\frac{t^{2}}{2\E[X]}\right)\,.
\end{split}
\end{equation}

The above Chernoff bounds, applied to the random variables $(D^+_x)_{x \in [n]}$, yields the following bounds on $\maxdeg$ and $\mindeg$ (maximum and minimum out-degree).

\begin{lemma} \label{C}
	There exists $C>1$ such that the event $\event:=\left\{ \mindeg \ge 2\right\} 
	\cap\left\{ \maxdeg \le C \log n \right\}$ satisfies
	\begin{equation}
		\p(\event)=1-o(1).
		\label{event}
	\end{equation}
\end{lemma}

\begin{proof}
	Fix a single vertex $x \in [n]$. It holds
	\begin{equation}
		\PP(D_x^+ < 2) = \prod_{y\neq x} (1-p_{xy})+ \sum_{z \neq x} p_{xz}\prod_{y \neq x,z}(1-p_{xy}), 
	\end{equation}
	and recalling that $\log(1-t)\le -t$ for every $|t|<1$,
	
	\begin{equation}
		\PP(D_x^+ < 2)\le e^{-\sum_{y\neq x} p_{xy}}+ \sum_{z \neq x} p_{xz}e^{-\sum_{y \neq x,z}p_{xy}}
		= O( n^{-w_x^+}\log n ).
	\end{equation}
	
	Since $w_x^+>1$ for every $x \in [n]$, by a union bound we get $\PP(\delta_x^+ < 2)=o(1)$. \\

	To bound below  $\maxdeg$, we apply the Chernoff bounds \eqref{binomial} to get
	\begin{equation}
		\PP(D_x^+ > C\log n)
		\le \exp\left( -\frac{(C\log n-\E[D_x^+])^2}{2\left(\E[D_x^+]+\frac{1}{3}(C \log n-\E[D_x^+] )\right)}  \right),
	\end{equation}
	and note that we can choose $C$ sufficiently large to obtain a uniform estimate in $x$, so that the \rhs is of order $n^{- \gamma}$, for any $\gamma >0$. 
Then, with a union bound on $x \in [n]$,
	\begin{equation}
		\PP(\maxdeg \le C\log n)=1-o(1).
		\qedhere
	\end{equation}
\end{proof}

\begin{lemma} \label{c} 
There exists a constant $c>0$, independent of $n$, such that, for every vertex $x \in [n]$,
	$$\PP(D_x^+ \le c\log n)=o(1)\,.$$
\end{lemma}
\begin{proof}
Applying the Chernoff bounds \eqref{binomial} with $X=D_x^+$ and $t:= \E[D^+_x]-c \log n>0$ 
	it holds
	\begin{equation} \label{c-bound}
		\PP(D_x^+ \le c\log n)
		\le \exp \left(- \frac{(\E[D_x^+]-c\log n)^2}{2\E[D_x^+]}\right).
	\end{equation}
	By assumption \eqref{eq1}, for every $x \in [n]$ it holds $\E[D_x^+]=\Theta(\log n)$, 
	with asymptotic constant uniformly bounded in $n$. 
Then, there exists \mbox{$c>0$}, independent of $n$, such that 
$$\frac{1}{\log n} \cdot\frac{(\E[D_x^+]-c\log n)^2}{2\E[D_x^+]}
	=\Theta(1)\, ,\quad 
		\forall x \in [n].$$
	This completes the proof.
\end{proof}
\begin{remark} 
In general, to perform a union bound in \eqref{c-bound} and prove that $\mindeg>c \log n$ 
w.h.p., it must hold, for $x \in [n]$,
$$\frac{1}{\log n} \cdot\frac{(\E[D_x^+]-c\log n)^2}{2\E[D_x^+]}= \alpha(x)(1+o(1)),$$
for a constant  $\alpha(x)$ such that $\alpha(x)>1$ uniformly in $x \in [n]$ and $n\in\mathbb{N}$.
This can happen only if, for large $n$ and for every $x \in [n]$, ${({\w w_x^+}/{n}-c)^2}>{2{\w w_x^+}/{n}}$. 
Since for every $n \in \N$, $c\in(0,{\w w_x^+}/{n})$, passing to the roots we derive the equivalent 
condition that $c < {\w w_x^+}/{n}-\sqrt{2{\w w_x^+}/{n}}$ for large $n$ and for every $x \in [n]$.

However, this condition is not always satisfied under our general hypotheses. 
For instance, on the \er graph with parameter $\lambda \log n/n$, where $1<\lambda<\sqrt{2}$, 
it holds that ${\w w_x^+}/{n}\equiv \lambda$, and the above condition is satisfied only if $c$ 
is such that $0<c<\lambda-\sqrt{2\lambda}<0$, yielding a contradiction. 
The above strategy is then insufficient to deal with this specific case.  
\end{remark}
\begin{remark} \label{inverse-remark}
The Chernoff bounds \eqref{binomial} provide a precise estimate on the average of the reciprocal of out-degrees. To see this, it is sufficient to plug $X=D^+_x$ and $t=m\E[D^+_x]$  into \eqref{binomial}.  
Since $\E[D^+_x]=\Theta(\log n)$, the choice $m=1/\sqrt[3]{\log n}$ implies 
	\begin{equation}\E\left[\frac{1}{D^+_x}\right]=\frac 1 {\E[D^+_x]}\Big[1+O\Big(\frac{1}{\sqrt[3]{\log n}}\Big)\Big].
	\end{equation}
\end{remark}

Notice that, thanks to Jensen's inequality, the multiplicative error term has to be greater than $1$. 
We conclude this subsection providing the proof of Proposition~\ref{prop-entropy} about the entropy $\entropy$. It is a straightforward application of the two previous lemmas.
\begin{proof}[Proof of Proposition~\ref{prop-entropy}]
From the definition of the entropy $\entropy$ given in \eqref{entropy}, we can conveniently rewrite
\begin{equation} \label{E[L]}
	\entropy=\sum_{x \in [n]}\muin(x) \sum_{i=2}^n \log i \p(D_x^+=i)\,.
\end{equation}
By Lemmas~\ref{C}-\ref{c}, for every fixed vertex \mbox{$x \in [n]$},
\begin{equation*}
	\p(D_x^+>C\log n)=o(1/n), \quad \quad \p(D_x^+ < c\log n)=o(1),
\end{equation*}
where $C>1$ and $c=c(x)>0$ uniformly in $n$. Hence
\begin{equation*}
	\log(c\log n) + o(1)	\le \sum_{i=2}^n \log i\p(D_x^+=i) \le \log(C\log n) + o( 1/n),
\end{equation*}
which together \eqref{E[L]}, implies that
${\entropy}={\log\log n} (1+o(1)).$

From the definition of the variance $\sigma^2$ given in \eqref{variance-sigma}, we can write
\begin{equation}
	\sigma^2 = \sum_{x \in [n]}\muin(x) \sum_{i=2}^n (\log i)^2 \p(D_x^+=i) - \entropy^2\,.
\end{equation}
Since for every $C \in (0+\infty)$ it holds 
$(\log(C \log n))^2=(\log\log n)^2 + 2 \log C \log\log n + \log^2 C$, 
from the previous displays, and inserting the derived estimate of $\entropy$, we conclude that 	
$\sigma^2 = O(\log\log n)$.
\end{proof}

The entropy $\entropy$ provides an average observable of the system. 
In the forthcoming sections it will be shown to be deeply connected with the dynamics of the random walk.
More precisely, we will deduce from Theorem \ref{thm_LLN} that the probability mass of a typical random walk path 
of length $t$ is  $e^{-\entropy t+O(\sqrt{\entropy t})}$.

\subsubsection{\textbf{Size of in-neighbourhoods}}
We now focus on the analysis of the in-neighbourhood properties of the graph, 
that will turn to be fundamental in understanding the spread of the random walk on the environment.

Recall that for $x\in[n]$ and $s\in\mathbb{N}$, $\mathcal{B}_x^+(s)$ and $\mathcal{B}_x^-(s)$
denote, respectively, the out- and in-neighbourhood of $x$ with depth $s$.
Following the general proof strategy traced in \cite{CCPQ}, we are going to show that, 
\whp and uniformly in $x$, the size of an in-neighbourhood of radius $\eps \tent/20$ is at most $n^{1/2+\eps}$. 
\begin{lemma}\label{size-in-ball}
	Let $h_\eps=\frac{\eps\log n}{20 \entropy}$  as in \eqref{h_eps}, and define the event
		\begin{equation} \label{S-}
			\mathcal{S}^-_\eps:=\{\forall x \in [n],|\mathcal{B}_x^-(h_\eps)|\le n^{1/2+\eps}\}.
		\end{equation}
	 Then $\p\left( \mathcal{S}^-_\eps\right) = 1-o(1)$.
		
		\end{lemma}
\begin{proof}
The idea is to provide a suitable upper bound on 
$\p(|\mathcal{B}_x^-(h_\eps)| > n^{1/2+\eps})$, and then conclude the proof by a union bound.
In this spirit, we claim  that, for $n$ large enough,	
\begin{equation}\label{claim1}
\E[|\mathcal{B}_x^-(h_\eps)|^2 \cdot \one_{\event} ]\le  w_x^- n^\eps \log^3 n\,,
		\end{equation}
where $\event$ is the typical event described in Lemma \ref{C}.
Assuming its validity, we readily get, by Markov's inequality, that
		\begin{equation}
			\p(|\mathcal{B}_x^-(h_\eps)| > n^{1/2+\eps}\,,\, \event)
			\le \frac{\E[|\mathcal{B}_x^-(h_\eps)|^2\cdot \one_{\event}]}{n^{1+2\eps}}
			\le  \frac{ w_x^- \log^3 n}{n^{1+\eps}}\,.
		\end{equation}
From Lemma \ref{C}, applying a union bound on $x\in[n]$ and by the assumption \eqref{assumption}, we conclude that for large $n$ 
\begin{equation} \label{important}
\begin{split}
\p({\mathcal{S}^-_\eps}^c)
&= 
\p({\mathcal{S}^-_\eps}^c \cap \event)+o(1) 
\le \sum_{x\in[n]} \p(|\mathcal{B}_x^-(h_\eps)| > n^{1/2+\eps}\,,\, \event) + o(1)
\\
& \le 
\frac{\log^3 n}{n^{1+\eps}}\sum_{x \in [n]}w_x^-+o(1)
=o(1)\,,
\end{split}
\end{equation}
which proves the statement.

It remains to show inequality \eqref{claim1}.		
Let $\mathcal{B}_x^\pm=\mathcal{B}_x^\pm({h_\eps})$ and write

\begin{equation}
\E[|\mathcal{B}_x^-|^2\cdot \one_{\event}]
=\sum_{y \in [n]}\sum_{z \in [n]}
\p(x \in \mathcal{B}_y^+,x \in \mathcal{B}_z^+, \event)\,,
\end{equation}
where	
	\begin{equation} \label{decomposition}
			\p(x \in \mathcal{B}_y^+,x \in \mathcal{B}_z^+, \event) 
			\le \p(x,z \in \mathcal{B}_y^+, \event)+\p(x \in \mathcal{B}_y^+,x \in \mathcal{B}_z^+, z \notin \mathcal{B}_y^+, \event).
		\end{equation}
We start by estimating the first term on the \rhs of the last display.
Note that, from the independence of the edge connectivity and applying Lemma \ref{C}, 
we can write
$$\p(x,z \in \mathcal{B}_y^+, \event)=\p(x\in \mathcal{B}_y^+, \event)\p(z \in \mathcal{B}_y^+| \event)  
= \p(x\in \mathcal{B}_y^+, \event)\p(z \in \mathcal{B}_y^+, \event)(1+o(1))\,,$$
and it is then enough to bound $\p(x \in \mathcal{B}_y^+\,,\,\event )$ for general $x\in[n]$. 

On the event $\event$, the out-neighbourhood $\mathcal{B}_y^+$ contains at most $(C\log n)^{h_\eps}$ vertices. 
Moreover, the probability that a vertex $u \in [n]\setminus \{x\}$ 
is connected to $x$ is 
$$p_{ux}=w_u^+w_x^- \frac{\log n}{n} \le M_1 w_x^- \frac{\log n}{n},$$ where $M_1$ 
is the constant given in the assumption \eqref{eq1}.	
Let $A_x$ denote the subset of $[n]$, of size $(C\log n)^{h_\eps}$, whose vertices maximize the parameters 
$(p_{u x})_{u \in [n]\setminus \{x\}}$. 
Then, for large $n$,
		\begin{equation} \label{bound}
			\begin{split}
				\p(x \in \mathcal{B}_y^+ \,,\,\event) 
				& \le
				\p\Bigg(\bigcup_{u \in \mathcal{B}_y^+\setminus\{x\}} \{u \to x\}\cap\event\Bigg)
				\le \p\left(\bigcup_{u \in A_x} \{u \to x\} \cap\event\right)\\
				&\le (C \log n)^{h_\eps} M_1 w_x^- \frac{\log n}{n}\le   w_x^-  n^{\frac{\eps}{10}}\frac{\log n}{n}\,.
			\end{split}	
		\end{equation}
We now bound the second term in \eqref{decomposition}. 
Note that, given that $x\in \mathcal{B}^+_y$ and $z\notin \mathcal{B}^+_y$, 
the event $x\in \mathcal{B}^+_z$ can be obtained if either $x$ is the closest 
vertex to $y$ in $\mathcal{B}^+_y\cap\mathcal{B}^+_z$, or there exists 
$u\neq x$ which is the closest vertex to $y$ in $\in \mathcal{B}^+_y\cap \mathcal{B}^+_z$ and that is connected to $x$ by a directed path.

Reasoning as before, and for large $n$, the first scenario  has probability less than 
$(w_x^- n^\frac{\eps}{10}\frac{\log n}{n})^2$,
while the second scenario is included in the event 
$E_{y,z,u}=\{u\in \mathcal{B}^+_y\cap \mathcal{B}^+_z\}\cap\{x\in \mathcal{B}^+_u\}$
that has probability
\begin{equation*}
\p(E_{y,z,u} \cap\event)\le  w_x^-(w_u^-)^2 n^\frac{3\eps}{10}\frac{\log^3 n}{n^3} \,.
\end{equation*}
All in all, and by assumption \eqref{assumption}, we get
\[
\begin{split}
\p(x \in \mathcal{B}_y^+,x \in \mathcal{B}_z^+, z \notin \mathcal{B}_y^+ \,,\,\event) 
&\le 
w_x^-w_z^- n^\frac{\eps}{5}\frac{\log^2 n}{n^2}
+M_2 w_x^- n^\frac{3\eps}{10}\frac{\log^3 n}{n^2}\,.  
\end{split}
\]
Summing over $y,\,z\in[n]$, and using that $\w=\Theta(n)$, we get that for large $n$
\begin{equation}
\E[|\mathcal{B}_x^-|^2\cdot \one_{\event}] \le w_x^-  n^\eps \log^3 n\,,
\end{equation}
which concludes the proof of the claimed inequality \eqref{claim1}, and then  of the lemma.
 \end{proof}

\subsubsection{\textbf{Tree excess of out-neighbourhoods}}\label{sec_excess}
Following \cite{BCS19}, we introduce a quantity that measures how much subgraphs look like trees.
Given a graph $S=(V,E)$, we  define its tree excess $\trex(S)$ as the minimum number of edges 
to remove in order to obtain a directed tree, that is 
$$\trex(S):=1+|E|-|V|\,. $$	
Then, for every $s\ge0$, we define the \textbf{bad} event $\mathcal{G}^+(s)$ as the set of graphs such that
there exists a vertex having an out-neighbourhood of depth $s$ with tree-excess greater than $1$, that is 
$$\mathcal{G}^+(s):=\bigcup_{x \in [n]}\{\trex(\mathcal{B}^+_x(s))\ge 2 \}\,.$$
\begin{lemma} \label{G+}
Let $h_\eps$ be as in \eqref{h_eps}. Then, for all $\eps$ sufficiently small, it holds
	\begin{equation}
				\PP(\mathcal{G}^+(2h_\eps))=o(1)\,.
	\end{equation}
\end{lemma}

\begin{proof}
    First note that, for any $x\in[n]$, the event $\{\trex(\mathcal{B}^+_x(s))\ge 2 \}$ corresponds to the event 
	that, while drawing iteratively $\mathcal{B}^+_x(s)$, at least two vertices are explored at least twice. 

	Let $C>1$ be a constant such that $\p(\event)=1-o(1)$, as in Lemma~\ref{C}, so that,
	being $\{\Delta\leq C\log n\}\subset \event$, it holds that 
	\begin{equation}
		\p(\mathcal{G}^+(2h_\eps))= \p(\mathcal{G}^+(2h_\eps) \cap \{\Delta_+\leq C\log n\}) + o(1).
	\end{equation} 
	On the event $\{\Delta_+\leq C\log n\}$, the ball $\mathcal{B}^+_x(s)$ has size  at most 
	$(C\log n)^{2h_\eps}$, and hence the probability of the event $\{\trex(\mathcal{B}^+_x(s))\ge 2 \}$ 
	can be bounded above using, as a counter of vertices which are explored at least twice, 
	a binomial random variable $\text{Bin}(m,q)$, where 	$m=(C\log n)^{2h_\eps}$ is 
	the maximum size of $\mathcal{B}^+_x(s)$, and $q$ bounds above the maximum probability of choosing an already explored vertex.
	
	In particular, letting $p_{\text{max}}:=\max_{x,y \in [n]} p_{xy}$ and with a union bound on the vertices $y\in \mathcal{B}^+_x(s)$,
	  we set $q=(C\log n)^{2h_\eps} p_{\text{max}}$ and get
		\begin{equation}
			\begin{split}
				\p(\trex(\mathcal{B}^+_x(s))\ge 2\,,\, \Delta_+<C\log n) 
				&\le  \p\left(\bin\left(m,q\right)\ge 2\right)\\
				&\le \left((C \log n)^{4h_\eps}p_{\text{max}}\right)^2\,.
			\end{split}
		\end{equation}
Since $p_{\text{max}}=O(n^{-\frac12 - \frac \eta 7})$, due to \eqref{maxp_xy}, and inserting the explicit value of $h_\eps$,
the \rhs of the above inequality turns to be $O(n^{-1+\frac {4 \eps} 5-\frac {2\eta}{7}})$. 
Choosing $\eps$ sufficiently small, e.g.~such that $\frac  {4 \eps} 5 < \frac{2}{7}\eta$, we conclude the proof by a union bound over $x\in[n]$.
\end{proof}

\section{Typical mass of random walk trajectories}
%
Having at hand some remarkable properties of the random environment, we switch to consider their impact on the random walk trajectories.
The goal of this section is to characterize the typical 
mass of a random walk of length $t=\Theta(\tent)$.
In particular, Theorem~\ref{thm_LLN} below can be interpreted as a 
quenched law of large numbers for this quantity (or rather its logarithm).
This last result will be then refined to a central limit theorem, which applies to
all trajectories of length $t$, with $t$ taken in an appropriate critical window 
(see Theorem \ref{CLT} below).\\

We start with a  simple lemma, that is a direct adaptation of Lemma 3.1 in \cite{CCPQ} 
and that will be useful in the next computations.
Recall the definition of the  vertex-set $V_\eps$ given in \eqref{veps}, whose elements are called 
$h_\eps$-roots. We are going to show that \whp with respect to the graph setting, the quenched probability that the random walk does not belong to $V_{\varepsilon}$ after $t$ steps decays at least exponentially in $t$.
%
\begin{lemma} \label{lemmaveps}  
	Let $h_\eps$ be as in \eqref{h_eps}. Then, for all $\eps$ sufficiently small and all $t \le h_\eps$, 		
	\begin{equation}
		\p(\max_{x \in [n]} \quench_x(X_t \notin V_\eps) \le 2^{-t})=1-o(1).
	\end{equation}
\end{lemma}
\begin{proof}
First note that, in the notation introduced in Subsection \ref{sec_excess}, we can rewrite
 $$\veps=\{y \in [n] : \trex(\mathcal{B}_y^+(h_\eps))=0\}.$$ 
In particular, due to Lemma~\ref{G+}, we can restrict ourselves,  with an error of order $o(1)$,
to the event 
$$(\mathcal{G}^+(2h_\eps))^c=\bigcap_{x\in[n]}\{\trex(\mathcal{B}_x^+(2h_\eps))\le 1\}\,.$$
In other words, under this event, the out-neighbourhood $\mathcal{B}_x^+(2h_\eps)$ is a directed tree except 
for at most one directed edge, for all $x\in[n]$. 
If $\mathcal{B}_x^+(2h_\eps)$ is a tree, then also $\mathcal{B}_{X_t}^+(h_\eps)$ is a tree and hence $X_t \in V_\eps$.
If $\mathcal{B}_x^+(2h_\eps)$ is not a tree, then it contains precisely one cycle and  we can identify 
the closest node to $x$ on this cycle, say $y$, that will be at a distance $s<2h_\eps$ from $x$.
Note that if $s<t$, then necessarily $\mathcal{B}_{X_t}^+(h_\eps)$ is a tree, as the contrary 
would imply the existence of a second cycle in $\mathcal{B}_x^+(2h_\eps)$, which is impossible under
$(\mathcal{G}^+(2h_\eps))^c$.
Instead, if $t\le s$, the event $\{X_t \notin V_\eps\}$ is realized only if the random walk
follows the unique directed path from $x$ to $y$ for $t$ steps.
 In view of Lemma~\ref{C}, we can further restrict  on the event $\event$, which ensures that $\mindeg \ge 2$,
  and on this event we derive the bound $\quench_x(X_t \notin V_\eps) \le 2^{-t}$, that holds w.h.p.~and concludes the proof. 
 \end{proof}

Before stating and proving the main results of this section, let us introduce some notation.

Let  $(D_k)_{k\ge 1 }$ be independent copies of $D_V^+$, the
random out-degree of a random vertex $V\in [n]$ sampled from $\muin$. This sequence is
defined w.r.t.~a probability measure that with a little abuse of notation will be simply denoted by $\p$.
Moreover, for $t\in\mathbb{N}$, set 
$$S_t:=\sum_{k=1}^t L_k\,,\quad  \mbox{ where }\quad L_k:=\log(D_k\vee 1)\,.$$
Then, for every $\theta \in (0,1)$ and $t\in\mathbb{N}$, we define 
\begin{equation}\label{q-secondo}
	q_t(\theta):=\p\left(\prod_{k=1}^t \frac{1}{D_k \vee 1}>\theta\right)=
	\p\left({S_t}<-\log (\theta)\right).
\end{equation}
Note that $q_t(\theta)$ corresponds to the probability that a path made of $t$ i.i.d.~samples from the in-degree distribution has mass at least $\theta$. 
Under suitable hypotheses, we will show that  the quenched probability
 $\mathcal Q_{x,t}(\theta)$, given in \eqref{quenched-theta}, is well approximated by $q_t(\theta)$. 
This is the crucial idea in order to prove the next result.
\begin{theorem}[Quenched Law of Large Number] \label{thm_LLN}
	Let $\mathcal Q_{x,t}(\theta)$ be the quenched  probability given in \eqref{quenched-theta},
and assume that $t=\Theta(\tent)$ and $\theta \in (0,1)$ are such that	
	\begin{equation}	\label{limitrho}
	-\frac{\log \theta}{\entropy t}\xrightarrow{n \to +\infty}\rho.
	\end{equation} 
	Then
	\begin{enumerate}
			\item[(i)] If $\rho < 1		\quad \Longrightarrow\quad
		\max_{x \in [n]}\cQ_{x,t}(\theta)\xrightarrow{ \quad\p\quad}0\,;$	\\
		\item[(ii)] If $\rho > 1\quad \Longrightarrow\quad 
			\min_{x \in [n]}\cQ_{x,t}(\theta)\xrightarrow{ \quad\p\quad}1\,.$
	\end{enumerate}
\end{theorem}

Note that, since $\entropy t=\Theta(\log n)$, the assumption \eqref{limitrho} implies that $\log \theta=\Theta(\log n)$. 
A possible choice could be $\theta=n^{-\rho}$, with possible multiplicative poly-log corrections.
\begin{proof}
	Our proof follows the strategy  given in \cite[Prop 3.2]{CCPQ}. 
	For $\ell=3\log \log n$, we define
	\begin{equation}
		\bar{\cQ}_{x,t}(\theta):=\sum_{y \in [n]}P^{\ell}(x,y)\cQ_{y,t}(\theta).
	\end{equation}
	This has the following interpretation. The first $\ell$ steps do not affect the total mass of the trajectory, 
	but in view of Lemma \ref{lemmaveps}, they are sufficient to let the walk move \whp to a $h_\eps$-root vertex. 
	Hence, we let the random walk move for $\ell$ steps and then start recording the mass 
	of the trajectory.
	For $\eps \in (0,\eta/2)$, with $\eta \in (0,1)$ as in \eqref{assumption}, we claim that
		\begin{equation} \label{bar}
		\max_{x \in V_\eps}|\bar{\cQ}_{x,t}(\theta)-q_t(\theta)|\xrightarrow{\quad \p \quad}0.
	\end{equation}
Before proving the claimed convergence,  we explore the asymptotic properties 
of $q_t(\theta)$, and then we complete the proof assuming the validity of \eqref{bar}.
As a first step, note that since $\{L_k\}_{k \ge 1}$ are i.i.d., 
and in view of Proposition \ref{prop-entropy}, it holds that
$$\\
\mathbb{E}(S_t)=\entropy t =\log n(1+o(1))\,,\qquad
\var(S_t)=\sigma^2 t = O(\log n)\,.
$$

From the hypothesis \eqref{limitrho}, it turns that  $-\log\theta= \rho \E[S_t] (1+o(1))$, 
so that we may expect the event in the definition of $q_t(\theta)$ to be typical or rare according to the value of $\rho$. 
Formally:
	\begin{enumerate}
		\item[(i)] if $\rho>1$ then, for large $n$, it holds $- \log\theta-\E[S_t]>0$ and
		\begin{equation}
			\begin{split}
				1-q_t(\theta)&=\p\left(S_t \ge -\log\theta\right)=\p\left(S_t-\E[S_t] \ge -\log\theta-\E[S_t]\right)
			\end{split}
		\end{equation}
		\item[(ii)] if $\rho<1$ then, for large $n$, it holds $ \log\theta+\E[S_t]>0$ and
		\begin{equation}
			\begin{split}
				q_t(\theta)=&\p\left(S_t< -\log\theta\right) =
				\p\left(-S_t + \E[S_t]\ge \log\theta+\E[S_t]\right)
			\end{split}
		\end{equation}
	\end{enumerate}
	In both cases, we can bound above the expression on the right-hand side of the last two displays
	by Chebyshev's inequality, and get
\begin{equation}\label{q(theta)}
\begin{split}
&
\p\left(|S_t-\E[S_t]| \ge |\log\theta+\E[S_t]|\right) 
\le \frac{\var(S_t)}{\left(\log\theta+\E[S_t]\right)^2}=o(1)
\\
&\Longrightarrow \quad
	q_t(\theta)\xrightarrow{ \quad\p\quad} 
\left\{ 
\begin{array}{ll}
1 & \mbox{ if } \rho>1\\
0 & \mbox{ if } \rho<1
\end{array}
\right. \,.
\end{split}
\end{equation}
Going back to the proof of our main statement, let us first observe that since the mass 
of a path of length $\ell$ is always in $[\Delta_+^{-\ell},\delta_+^{-\ell}]$, it holds  that
\begin{equation}\label{dettaglio1}
\begin{split}
\cQ_{x,t}(\theta)
 &\le \quench_x(\mass(X_\ell,X_{\ell +1},\dots,X_t)>\theta \delta_+^{\ell})
 = \bar{\cQ}_{x,t-\ell}(\theta \delta_+^\ell) \\
 &\le \quench_x(\mass(X_\ell,X_{\ell +1},\dots,X_t)>\theta \delta_+^{\ell}| X_\ell \in V_\eps)
+ \quench_x(X_\ell \notin V_\eps)\\
& \le \max_{y\in \veps} \cQ_{y,t-\ell} (\theta\delta_+^\ell) + \quench_x(X_\ell \notin V_\eps)\\
& \le \max_{y\in \veps} \bar{\cQ}_{y,t-2\ell} (\theta\delta_+^{2\ell}) + \quench_x(X_\ell \notin V_\eps)\\
& \le \max_{y\in \veps} \bar{\cQ}_{y,t} (\theta\delta_+^{2\ell}\Delta_+^{-2\ell}) + \quench_x(X_\ell \notin V_\eps) \,.
\end{split}
\end{equation}
By Lemma \ref{lemmaveps} and assuming the validity of \eqref{bar}, we get that
\begin{equation} 
  \begin{split}
  \max_{x \in [n]}\cQ_{x,t}(\theta) 
    &\le q_t(\theta \delta_+^{2\ell}\maxdeg^{-2\ell})+o_{\p}(1)\,.
  \end{split}
\end{equation}
Since $q_t(\cdot)$ is decreasing, and both \whp $\maxdeg\le C \log n$ and $\delta_+ \ge 2$ are valid, we conclude that 	
	\begin{equation} \label{3.8}
		\max_{x \in [n]}\cQ_{x,t}(\theta)\le q_t(\theta2^{\ell}(C\log n)^{-\ell})+o_{\p}(1).
	\end{equation}	
Similarly, we first observe that by definition
\begin{equation}\label{dettaglio2}
\begin{split}
\cQ_{x,t}(\theta)
 &\ge \quench_x(\mass(X_\ell,X_{\ell +1},\dots,X_t)>\theta \Delta_+^{\ell})
 = \bar{\cQ}_{x,t-\ell}(\theta \Delta_+^\ell) \\
 &\ge \quench_x(\mass(X_\ell,X_{\ell +1},\dots,X_t)>\theta \Delta_+^{\ell}| X_\ell \in V_\eps)
\quench_x(X_\ell \in V_\eps)\\
& \ge \min_{x\in \veps} \cQ_{x,t-\ell} (\theta\Delta_+^\ell)\quench_x(X_\ell \in V_\eps)\\
& \ge \min_{x\in \veps} \bar{\cQ}_{x,t-2\ell} (\theta\Delta_+^{2\ell})\quench_x(X_\ell \in V_\eps)\,.
\end{split}
\end{equation}
By Lemma \ref{lemmaveps} and assuming again the validity of \eqref{bar}, we obtain 
\begin{equation}
	\begin{split}
			\min_{x \in [n]}\cQ_{x,t}(\theta) 
			&\ge  \min_{x\in \veps}\bar{\cQ}_{x,t-2\ell}(\theta\maxdeg^{2\ell})(1-2^{-\ell}-o_{\p}(1))\\	
			&\ge q_{t-2\ell}(\theta\maxdeg^{2\ell})+o_{\p}(1)\,
			\ge q_t(\theta\maxdeg^{2\ell})+o_{\p}(1)\,.
		\end{split}
	\end{equation}
Since $q_t(\theta)$ is decreasing in $t$ and  $\maxdeg\le C \log n$ \whp, we conclude that
	\begin{equation} \label{3.10}
		\min_{x \in [n]}\cQ_{x,t}(\theta)\ge q_t(\theta(C\log n)^{2\ell})+o_{\p}(1).
	\end{equation}
		At last note that, setting $\theta'=\theta(C\log n)^{\pm2\ell}$, then 
		$\log{\theta'}= \log{\theta}+ O((\log \log n)^2)$. 
		 Since the asymptotic value of $q_t(\cdot)$ is not sensitive to perturbations $\theta'$ 
	such that $|\log\theta'-\log\theta|=O((\log \log n)^2)$, Eqs.~\eqref{3.8}-\eqref{3.10}, 
	together with \eqref{q(theta)}, conclude the proof of our statement.	
	\\

	Let us finally prove the claimed convergence \eqref{bar}. 
	To this aim, we are going to show that, for all $\delta >0$,
	\begin{equation}\label{fineprova}
		\p(\one_{x \in \veps}\bar{\cQ}_{x,t}(\theta)\ge q_{t}(\theta)+\delta)=o(n^{-1}),
	\end{equation}
	and then we apply a union bound over $x \in \veps$. 
	This will give only half of \eqref{bar}, but actually the same argument applies to 
	$1-\bar{\cQ}_{x,t}(\theta)$ and $1-q_t(\theta)$, completing the proof.\\
	
For any fixed $K\ge 1$, by Markov's inequality  we get
	\begin{equation}\label{due}
		\p\left(\one_{x \in \veps}\bar{\cQ}_{x,t}(\theta)\ge q_{t}(\theta)+\delta \right) 
		\le 
		\frac{\E\left[\one_{x \in \veps}\left(\bar{\cQ}_{x,t}(\theta)\right)^K \right]}
		{\left(q_{t}(\theta)+\delta\right)^K}\,.
	\end{equation}

We now follow the strategy of the proof given in \cite[Prop. 3.2]{CCPQ}.
Consider the annealed law $\p^{\textup{an},K}_x$ of the process $(X^{(k)})_{k \in \{1,\dots,K\}}$ 
defined in Subsection \ref{annealed random walk}, for $T=t+\ell$. 
The process consists of $K$ random walks of length $t+\ell$ and initial measure $\delta_x$, realized one after the other together with the partial graph structure that they explore. 
Let $K=\lfloor \log^2( n)\rfloor$ and, for every $1 \le j \le K$, define the event 
$B_j$ through the following conditions:
\begin{enumerate}
\item[(i)] the union of the first $j$ trajectories up to time $\ell$, that is 
$(X^{(1)}_s,\dots,X^{(j)}_s)_{s\le\ell}$, forms a directed tree;
\item[(ii)] for every $i \le j$, the last $t$ steps of the $i$-th walk, that is $(X^{(i)}_s)_{s\in[\ell+1,\ell+t]}$, 
define a path $\pat$ of mass $\mass(\pat) > \theta$;
\item[(iii)] The vertices in the first $j$ trajectories 
 have out-degree at least $2$.
\end{enumerate}	
By definition, note that the event $\{x \in \veps\}$ is contained 
in the event that the $K$ trajectories form a tree up to depth $\ell$. Hence
\begin{equation}\label{tre}
		\E[\one_{x \in \veps}(\bar{\cQ}_{x,t}(\theta))^K  ] \le
		\p_x^{\textup{an},K}(B_K )
		= \p_x^{\textup{an},K}(B_1)\prod_{i=2}^{K} \p_x^{\textup{an},K}(B_j \ | \ B_{j-1}).
\end{equation}
Note that, given $B_{j-1}$:
\begin{enumerate}
\item either the $j$-th walk follows one of the  previously traced trajectories up to time $\ell$,
thus keeping unchanged the tree structure of depth $\ell$ around $x$.
\item or  the $j$-th walk explores a new vertex before time $\ell$. In that case, the event $B_j$ takes place if the $j$-th walk keeps exploring new vertices at least up to time $\ell$, in order to preserve the whole tree structure, and then moves its last $t$ steps on a path $\pat$ with mass $\mass(\pat)> \theta$.
\end{enumerate}

Since the out-degree of these vertices is at least $2$ by the conditioning, 
the first scenario happens, for all $j\le K$, with conditional probability which is at most 
$$(K-1)2^{-\ell}\le K2^{-\ell}=e^{2\log\log n - \ell\log 2}=o(1).$$

To estimate the probability of the second scenario, first note that, at each step, the conditional probability to visit 
an already explored vertex is less than $K (t+\ell) p_{\text{max}}$.  
Summing this term for all the $\ell+t$ steps of the path, we obtain
that the conditional probability that the $j$-th walk visits an already explored vertex,
and create a cycle along the whole process, is less than $(t+\ell)^2K p_{\text{max}}=o(1)$, for all $j\le K$. 
Hence the tree structure is preserved  w.h.p. along the whole trajectory.

Moreover, on the event that the $j$-th trajectory always visits new vertices, 
the conditional law of its last $t$ steps corresponds to the annealed law 
of a random walk of length $t$ defined on a reduced Chung--Lu graph, 
which is obtained by removing the vertices explored by the whole process before its last $t$ steps,
on the event that it has no self-intersections. 
In particular, from Lemma \ref{lemmaiid} and Eq.~\eqref{tv-cost}, each step of this random walk can be chosen approximately as a sample of $\muin$.  
In other words, after exiting the already visited trajectories, the rest of the path up to step 
$t+\ell$, can be coupled with an \iid sample from $\muin$  with an overall total variation cost which is of order 
$O((t+\ell)^2K p_{\text{max}})=o(1)$.
The second scenario is  then satisfied with probability $q_t(\theta)+o(1)$. 	\\

Altogether, this shows that, for all $\delta>0$ and for all $j\le K$, 
$$\p_x^{\textup{an},K}(B_j \ | \ B_{j-1})\le
q_{t}(\theta)+ \frac{\delta}{2}\,,$$ 
that, thanks to Eqs.~\eqref{due}-\eqref{tre},
implies \eqref{fineprova}. This ends the proof of the claimed convergence \eqref{bar} and of the theorem. 
\end{proof} 

Let us now consider a time window of size $\textbf{\textup{w}}_n:=\frac{\sigma}{\entropy}\sqrt{\tent}$, as given in \eqref{window}.
Then it holds the following.
\begin{theorem}[Central Limit Theorem] \label{CLT}
Let $t_{\lambda}:=\tent + \lambda\textbf{\textup{w}}_n + o(\textbf{\textup{w}}_n)$,  
with $\lambda \in \mathbb{R}$ fixed, and assume that $\theta\in(0,1) $ is such that
	\begin{equation}\label{lambda}
		\frac{\log \theta+\entropy t_{\lambda}}{\sigma \sqrt{t_{\lambda}}} \xrightarrow[n \to +\infty]{} \lambda,
	\end{equation}
where $\sigma^2$ satisfies the non-degeneracy condition \eqref{non-degenerate-variance}.
Then
	\begin{equation} 
		\max_{x \in [n]} \Big| \cQ_{x,t_{\lambda}}(\theta)- \frac{1}{\sqrt{2\pi}}\int_{\lambda}^\infty e^{-u^2/2}\,du \Big| \pconv 0.
	\end{equation}
\end{theorem}
Note that, since $t_{\lambda}= \tent(1+o(1))$ and $\entropy t_{\lambda}=\log n +\lambda \sigma \sqrt{\tent}$, 
the assumption \eqref{lambda} implies that $\log \theta=-\log n(1+o(1))$. 
A possible choice could be $\theta=n^{-1}$, with possible multiplicative poly-log corrections.

\begin{proof}
To ease the notation, let $t=t_{\lambda}$.
In view of the convergence \eqref{bar}, we first focus  on the probability $q_t(\theta)$. By Eq.~\eqref{q-secondo}, we can write
	\begin{equation}
		\begin{split}
			q_t(\theta)= \p\left(\frac{S_t- \entropy t}{\sigma\sqrt{t}} 
			< -\frac{\log(\theta)+\entropy t}{\sigma\sqrt{t}}\right)\,.
		\end{split}
	\end{equation}
	Looking at the argument of that probability, while the \rhs converges to $-\lambda$ due to assumption \eqref{lambda},
	we will prove that the \lhs converges in distribution to a Normal. 
	We can indeed check that the Lyapunov condition of the Lindeberg-Feller Central Limit Theorem 
	holds (see, e.g., \cite{klenke}, Lemma 15.41).
	Specifically, we need to prove  that there exists $\delta>0$ such that 
	\begin{equation} \label{quotient}
		\lim_{n \to + \infty}\frac{1}{\textup{Var}(S_{t})^{1+\delta/2}} 	\sum_{k=1}^{t}\E[{|L_k-\entropy|}^{2+\delta}]=0\,.
	\end{equation}
We observe that, due to Lemmas \ref{C} and \ref{c}, and by our choice of $t$, for all $\delta>0$,
	\begin{equation}
		\sum_{k=1}^t\E[{L_k}^{2+\delta}] = t\E[{L_1}^{2+\delta}] 
		= \log{n} (\log \log n)^{1+\delta}(1+o(1))\,.
	\end{equation} 
Using that 
$\E[{|L_k-\entropy|}^{2+\delta}]\le 2^{1+\delta}\left(
\E[{|L_k|}^{2+\delta}]+ \entropy^{2+\delta}\right)$,
and thanks to Proposition \ref{prop-entropy}, we then get that
the numerator of \eqref{quotient} is  $O(\log{n} (\log \log n)^{1+\delta})$.

On the other hand, let $\delta>0$ be such that the non-degeneracy condition \eqref{non-degenerate-variance} on $\sigma^2$ is satisfied. Then
\begin{equation} \label{1}
		\var(S_t)^{1+\delta/2}= (t \sigma^2)^{1+\delta/2}\gg \log{n} (\log \log n)^{1+\delta}\,,
\end{equation}
and the 	Lyapunov condition \eqref{quotient} is  verified.
As a consequence, 	
\begin{equation} \label{gaussian}
		\lim_{n \to +\infty} q_t(\theta) = \frac{1}{\sqrt{2\pi}}\int_{-\infty}^{-\lambda} e^{-\frac{u^2}{2}}\,du 
		= \frac{1}{\sqrt{2\pi}}\int^{+\infty}_{\lambda} e^{-\frac{u^2}{2}}\,du.
	\end{equation}
The thesis now follows thanks to the convergence \eqref{bar}, together with the bounds \eqref{3.8} and \eqref{3.10}, and to the fact that  the asymptotic value of $q_t$ is not sensitive to perturbations 
$\theta'$ such that $|\log\theta'-\log\theta|=O((\log \log n)^2)$.
\end{proof}
	
\section{Proofs}

\subsection{Tree-like trajectories} 
The goal of this section is to analyse the random kernel of the random walk in order to prove 
that the properties characterizing nice paths, listed in Definition~\ref{def}, hold \whp as $n\to\infty$.
We will first show that, for all times $s\le (1-\gamma)\tent$, where $\gamma =\frac{\eps}{80}$ as in Definition \ref{def}, the random walk trajectories of length $s$ live \whp in the tree $\mathcal{T}_x(s)$ given in Subsection \ref{tree-constr}.
Accordingly to Remark \ref{rem-gen-tree}, this result can be extended with few little adjustments
to times $s=t_\lambda- h_\eps$, with $t_\lambda$ lying in the critical window of Theorem \ref{thm_window}. 
We will briefly comment at the end of the subsection.\\

We start with a preliminary result.
Recall the  notation introduced in Subsection \ref{sec_trees} and the procedure
to construct the tree $\mathcal{T}_x(s)\subset \mathcal{G}_x(s)$, which involves the sequences  of graphs 
$(\mathcal{G}^\ell)_{\ell \ge 0}$ and $(\mathcal{T}^\ell)_{\ell \ge 0}$, and the sequence of edges 
$(e_{\ell})_{\ell \ge 0}$. 
In particular, remind that
$\mathcal{T}_x(s):= \mathcal{T}^{\kappa_x}$, 
where $\kappa_x$ is the index of the last iteration of the algorithm. 
\begin{lemma}\label{facts} 
For all $1\le \ell\leq \kappa_x$, let $e_{\ell}$ denote the edge chosen by the $\ell$-th iteration 
of the algorithm defining $\mathcal{T}_x(s)$. Then, on the event $\event$, it holds that
		\begin{equation} \label{fact-a}
			e^{-\overh s}\le \hat{\mass}(e_{\ell}) \le \frac{2}{2+\ell}\, ,
		\end{equation}
	where $\hat{\mass}(e_{\ell})$ was given in \eqref{mass-hat}, and $\overh = (1+\gamma)\entropy$. 
		As a consequence, $\kappa_x \le 2 e^{\overh s}$.
\end{lemma}
\begin{proof}
See the proof of \cite[Lemma 11]{BCS18}, which applies to the present setting without substantial changes.
\end{proof}

With this result at hand, we can prove the following proposition, which shows that, w.h.p., a random walk starting 
from a vertex in  $V_\eps$ performs a trajectory in $\mathcal{T}_x(s)$.
To state the result, let us denote by $\mathcal{P}(x,y,s,H)$  the set of paths from $x$ to $y$ of length $s$, in a subgraph $H$ of $G$
. 
\begin{prop} \label{tree}
	For all $\eps$, $\gamma \in (0,1)$  and $s\le(1-\gamma)\tent$, it holds
	\begin{equation}
		\min_{x \in \veps}\left(
		\sum_{y \in [n]} \sum_{\pat \in \mathcal{P}(x,y,s,\mathcal{T}_x(s))} \mass(\pat)
		\right) \pconv 1.
	\end{equation}
\end{prop}
\begin{proof}
Note that,  by the definition of $\mathcal{T}_x(s)$, a path $\pat\in \mathcal{P}(x,y,s,G)$ 
is not in $\mathcal{P}(x,y,s,\mathcal{T}_x(s))$  if one of the two following conditions holds:
\begin{enumerate}
	\item[(1)] 
	 $\mass(\pat)\le e^{-\overh s}= 1/n^{1-\gamma^2}(1+o(1))$.
	\item[(2)] 
	$\pat$ has edges in $\mathcal{G}_x(s)\setminus \mathcal{T}_x(s)$.
\end{enumerate}
For $j=1,2$, we denote with $\mathcal{P}^{j,*}_{x,y}$ the set of paths in $\mathcal{P}(x,y,s,G)$ 
for which condition $(j)$ does not hold, and observe that by definition 
$$\sum_{y \in [n]} \sum_{{\pat \in \mathcal{P}^{1,*}_{x,y}}} \mass(\pat)\ge \mathcal Q_{x,s} (1/n^{1-\gamma^2})\,,\quad \forall x\in[n]\,.$$
Since $ \frac{\overh s}{\entropy s} = 1+\gamma > 1$, Theorem~\ref{thm_LLN}(ii) applies
and  we get that 
\begin{equation}\label{conv}
\min_{x \in [n]} \Bigg\{\sum_{y \in [n]} \sum_{{\pat \in \mathcal{P}^{1,*}_{x,y}}} \mass(\pat) \Bigg\}\pconv 1\,,
\end{equation}
which  proves that condition $(1)$ is not likely to be satisfied. 
To estimate the probability that condition (2) is  satisfied, let us define iteratively $(M_\ell)_{\ell=0}^{\kappa_x}$ setting
\begin{equation}
M_0:=0\,,\qquad	
M_\ell:=M_{\ell-1}+\hat{\mass}(e_{\ell})\one(\ell\le \kappa_x)
\one(v_{e_{\ell}}^+\in V(\mathcal{G}^{\ell-1})), 
\quad \forall \ell\in\{1,\ldots, \kappa_x\} \,,
\end{equation}
where $V(H)$ denotes the vertex set of a graph $H$ and $v_{e}^+$ denotes 
the head of an edge $e$.
Note that $M_\ell$ represents the total probability mass that is excluded from $\mathcal{G}^\ell$ in the generation of $\mathcal{T}^\ell$. We recall that $e_{\ell}$ is the edge selected by the $\ell$-th iteration of the algorithm.
In particular
	\begin{equation}
		M_{\kappa_x}=\sum_{y \in [n]}\sum_{\pat \in \mathcal{P}^{2,*}_{x,y}} \mass(\pat).
	\end{equation}
We want to show that, for all $\delta>0$,
\begin{equation}  \label{thesis}
	\p(\exists x \in V_\eps : M_{\kappa_x}>\delta)=o(1).
\end{equation}
To this aim, we first prove
\begin{equation}  \label{M-delta}
\p(M_{\kappa_x}>\delta,\event)=o(n^{-1}),
\end{equation}
so that Lemma \ref{C} and a union bound over $x \in \veps$ are sufficient to conclude the proof.  

Let $\ell_\eps= 2^{h_\eps} $. Remember that condition $(1)$ above is satisfied with vanishing probability for $s=h_\eps=\Theta(\tent)$ and observe that $(\hat{\mass}(e_\ell))_{\ell\ge 0}$ is decreasing in $\ell$. Moreover notice that, being $x \in V_\eps$, $\mathcal{T}_x(h_\eps)$ is a tree. Combining these facts, it follows that \whp $\mathcal{T}_x(h_\eps)=\mathcal{G}_x(h_\eps)=\mathcal{B}_x^+(h_\eps)$.
In conclusion, \whp, $\kappa_x=|\mathcal{T}_{\kappa_x}|\ge|\mathcal{B}_x^+(h_\eps)|\ge 2^{h_\eps}=\ell_\eps$.

As a by-product of the previous lines, we get that in the first $\ell_\eps$ steps of the construction of $\mathcal{T}_{\kappa_x}$, no mass is thrown away and then $M_\ell=0$ for all $\ell\le \ell_\eps$.
Moreover, due to \eqref{fact-a}, on the event $\event$
\begin{equation}	\label{Ml}
		M_\ell-M_{\ell-1}\le \frac{2}{2+\ell_\eps}
\le 2^{-h_\eps+1}\le 1 \,,\qquad \forall \ell \ge \ell_\eps+1\,.
\end{equation}
Let $\mathcal{F}_\ell$ denote the $\sigma$-field associated to the first $\ell$ 
generation steps of $\mathcal{T}_x(s)$. 
By the previous estimates, it turns out that
\begin{equation}\label{M2}
\E[(M_\ell-M_{\ell-1})\one_{\event} \,  | \,\mathcal{F}_{\ell-1}] 
\le 
\frac 2{2+\ell}\cdot\p(v_{e_{\ell}}^+\in V(\mathcal{G}^{\ell-1})\,,\,\event|\, \mathcal{F}_{\ell-1})\,,
\quad \p-a.s.\,,\qquad \forall \ell \ge \ell_\eps+1\,,
\end{equation}
where
\begin{equation} \label{conditional-p}
\p(v_{e_{\ell}}^+\in V(\mathcal{G}^{\ell-1})\,,\,\event |\, \mathcal{F}_{\ell-1})
\le
\max_{y\in V(\mathcal{G}^{\ell-1})}
\sum_{z \in V(\mathcal{G}^{\ell-1})} p_{y,z} \le 
M_1 \frac{\log n}{n} \sum_{z \in V(\mathcal{G}^{\ell-1})}w_z^-\,.
	\end{equation}
To estimate the r.h.s. of the display, note that, for any $S\subset [n]$
and taking $\zeta$ such that $6\zeta < \eta$, with $\eta$ as given 
in assumption \eqref{assumption}, 
we can apply H\"older's inequality and get, for $p=2+6\zeta$,
\begin{equation}\label{hoelder}
\sum_{z \in S} w_z^- \le \left[\sum_{z \in S} (w_z^-)^{p}\right]^{\frac1 {p}} |S|^{1-\frac 1{p}} 
\le \left[\frac{M_2n}{|S|}\right]^{\frac{1}{2+6\zeta}}|S|.
\end{equation}
We take $S=V(\mathcal{G}^{\ell-1})$ and observe that $\frac{1}{2+6\zeta}<\frac12 -\zeta$. 
Since $|V(\mathcal{G}^{\ell-1})|\le\kappa_x\le 2 n^{1-\gamma^2}$, 
where the last inequality is due to Lemma \ref{facts},
we obtain that
\begin{equation}
\sum_{z \in V(\mathcal{G}^{\ell-1})}w_z^- =o(n^{1 -\xi}),
\end{equation}
for $\xi>0$ sufficiently small, depending on the given $\zeta$ and $\gamma$. 
Inserting this estimate in \eqref{conditional-p} and then in \eqref{M2}, 
we conclude that, for any $\ell\le \kappa_{x}$,
$$
\E[(M_\ell-M_{\ell-1})\one_{\event}\, |\, \mathcal{F}_{\ell-1}]  = \frac{1}{\ell}\, o\left(\frac{\log n}{ n^{\xi}}\right)\,,
$$
and in a similar way that
$$
\E[\left(M_\ell-M_{\ell-1}\right)^2\one_{\event}\, |\, \mathcal{F}_{\ell-1}]  = 
\frac{1}{\ell^2}\, o\left(\frac{\log n}{ n^{\xi}}\right)\,.
$$
Consequently, 
	\begin{equation} \label{a}
		a:=\sum_{\ell=1}^{\kappa_x} \E[M_\ell-M_{\ell-1} \one_{\event}\,|\, \mathcal{F}_{\ell-1}]  
		= o\left(\frac{\log^2 n}{ n^{\xi}}\right)\,,
	\end{equation}
	\begin{equation} \label{b}
		b:=\sum_{\ell=1}^{\kappa_x} \E[\left(M_\ell-M_{\ell-1}\right)^2 \one_{\event}\,|\, \mathcal{F}_{\ell-1}] 
		=o\left(\frac{\log n}{ n^{\xi}}\right)\,,
	\end{equation}
	where we used the fact that  $\sum_{\ell=1}^{\kappa_x} \ell^{-1}=O(\log{\kappa_x})$. 
	For $\ell\in\{0,\ldots, \kappa_x\}$, we define  
	\begin{equation} \label{Z}
		Z_{\ell+1}:=\frac{c_\xi}{\delta}(M_{\ell+1}-M_{\ell}-\E[(M_{\ell+1}-M_{\ell})\one_{\event} | \mathcal{F}_{\ell}]) \one_{\event},
	\end{equation}
 where $c_\xi:=2/\xi+2$. Thanks to \eqref{Ml}, $|Z_{\ell+1}|\le 1$ for large $n$. 
  Since $\kappa_x \ge \ell_\eps$, we also define
	\begin{equation}
		\phi_u:=\sum_{i = \ell_\eps}^u Z_{i+1}\,,\qquad \forall u\in\{\ell_\eps,\ldots, \kappa_x\}\,.
	\end{equation}
The sequence $(\phi_u)_{\ell_\eps\le u \le \kappa_x}$ is a martingale.
Observe that $M_{\kappa_x}=a+ \frac\delta {c_\xi} \phi_{\kappa_x}$. 
Thanks to the estimates \eqref{a} and \eqref{b},  we can assume that
$a \le \frac{\delta}{c_\xi}$ for large enough $n$. 
Hence, recalling that $c_\xi-2=2/\xi$, we can write
		\begin{equation}\label{M-xi}
			\begin{split}
				\p\Big(M_{\kappa_x}\ge \frac {c_\xi-1}{c_\xi}\delta , \event\Big) 
				 \le 
				 \p\Big(\phi_\ell \ge \frac 2\xi \text{ for some }\ell 
				 \ge \ell_\eps\,,\, \event\Big)\,.
			\end{split}
		\end{equation} 
At last, let us consider the conditional variance $$b':=\sum_{i=1}^\ell \textup{Var}(Z_i | \mathcal{F}_{i}).$$
On $\event$, thanks to \eqref{b}-\eqref{Z}, 
$b'\le (c_\xi/\delta)^{2}b=o\left(\frac{\log n}{ n^{\xi}}\right)$ uniformly in $\ell$. 
Choosing $c(n)=\frac{\log n}{ n^{\xi}}$, for all $n\gg 1$ it holds
$$
\p\Big(b' > c(n)\text{ for some }\ell \ge \ell_\eps\,,\,\event \Big)=0 \,,
$$ 
and thus
\begin{equation}
\begin{split}
\p\Big(\phi_\ell \ge \frac 2\xi \text{ for some }\ell \ge \ell_\eps \,,\event \Big) 
=
\p\Big(\phi_\ell \ge \frac 2\xi, b' \le c(n) \text{ for some }\ell \ge \ell_\eps\,,\event \Big)\,.	\end{split}
\end{equation}
As in \cite[Lemma 3.3]{CCPQ}, we apply \cite[Theorem 1.6]{Freedman} to bound the \rhs 
with $$e^{\frac 2\xi}\left(\frac{c(n)}{\frac 2\xi+c(n)}\right)^{\frac 2\xi+c(n)}=o(n^{-1}).$$ 
Inserting this bound in \eqref{M-xi}, we obtain \eqref{M-delta} which concludes the proof.
\end{proof}
\begin{remark}\label{rem-fine}
The statement of this proposition can be easily generalized to time $s=t_\lambda- h_\eps$, with $t_\lambda$ 
lying in the critical window of Theorem \ref{thm_window}. 
Indeed, with this specific choice, it holds $s=(1-4\gamma)\tent(1+o(1))$ and $n^{1-4\gamma} \le e^{\overh s}\le n^{1-\gamma^2}$.
All the estimates involving $s$, and specifically Lemma \ref{facts} and the convergence \eqref{conv}, come true without substantial changes.
\end{remark}

\subsection{Proof of the upper bound on the mixing time (Eq. \eqref{main2} of Theorem \ref{thm_main})}\label{upper}

We start by rearranging in a more convenient form the total variation distance of the statement. 
For $h=h_\eps$ as in \eqref{h_eps}, let 
	\begin{equation} \label{tildaPi}
		\widetilde{\pi}:=\muin P^{h}\, ,
	\end{equation}   
and write, by the triangle inequality,
	\begin{equation}
		\|P^{t}(x,\cdot)-\pi\|_{\textup{TV}} 
		\le \|P^{t}(x,\cdot)-\widetilde{\pi}\|_{\textup{TV}}+\|\widetilde{\pi}-\pi\|_{\textup{TV}}.
	\end{equation}
If the first term in the \rhs is $o_{\p}(1)$ uniformly in $x \in [n]$, then by the triangle inequality and \eqref{triangle-ineq}, the same must hold for the second term. 
Let $t'=t+\ell$, with $\ell=\log \log n$ and $t=(1+\beta)\tent$.
 Applying the Markov property,
	\begin{equation}
		\begin{split}
		\max_{x \in [n]}\|P^{t'}(x,\cdot)-\widetilde{\pi}\|_{\textup{TV}} 
		&= \max_{x \in [n]}  \| \sum_{y \in [n]} P^{\ell}(x,y) (P^{t}(y,\cdot)-\widetilde{\pi}\, ) (\one_{\{y \in V_\eps\}}+\one_{\{y \notin V_\eps\}})\|_{\textup{TV}} \\
		&\le \max_{x \in [n]}  \frac 12 \sum_{z \in [n]} \sum_{y \in V_\eps}
		P^{\ell}(x,y) |P^{t}(y,z)-\widetilde{\pi}(z) |
		+\max_{x \in [n]}  \frac 12  \sum_{y \in [n]\setminus V_\eps} 2P^{\ell}(x,y)\\
		& \le \max_{y \in \veps} \|P^{t}(y,\cdot)-\widetilde{\pi}\|_{\textup{TV}} + \max_{x \in [n]} \quench_x(X_\ell \notin \veps), 
		\end{split}
	\end{equation}
	where the first inequality is obtained by the triangle inequality and bounding the total variation distance by $2$, while the second inequality is obtained by changing the order of the sums
and maximizing over $y$ (so that $x$ disappears).
	
	The second term is arbitrarily small due to Lemma \ref{lemmaveps}. 
	We then focus on the first term. For sake of readiness we will keep calling $x$ the maximizing variable and we will bound $\max_{x \in \veps} \|P^{t}(x,\cdot)-\widetilde{\pi}\|_{\textup{TV}}$.
	For every $x,y\in[n]$, let $\widetilde{P}^t(x,y)$ be the probability to go from $x$ to $y$ in $t$ steps 
	following a nice path, that is
\begin{equation}\label{Ptilde}
\widetilde{P}^t(x,y)= \sum_{\pat\in\widetilde{\mathcal{P}}(x,y,t,G)} \mass(\pat)\,,
\end{equation}	
where $\widetilde{\mathcal{P}}(x,y,t,G)$ is the set of nice paths from $x$ to $y$ of length $t$ in $G$. Moreover we set
\begin{equation}
\widetilde{q}(x):=1-\sum_{y \in [n]} \widetilde{P}^t(x,y)\,.
\end{equation}

	Then, for all $\delta>0$, it holds 	
	\begin{equation} \label{eq2}
			\begin{split}
				\|P^{t}(x,\cdot)-\widetilde{\pi}\|_{\textup{TV}}\le
				\sum_{y \in [n]}\left[(1+\delta)\widetilde{\pi}(y)+\frac \delta n-\widetilde{P}^t(x,y)\right]^+.
			\end{split}
	\end{equation}

To handle the term in the r.h.s. above, we apply Proposition \ref{drop-positive-part} below 
in order to remove the positive part $[u]^+=\max\{0,u\}$ in \eqref{eq2}.
From this statement, we get that for all $\delta>0$, and w.h.p., \eqref{eq2} becomes 
\begin{equation}\label{eq3}
\begin{split}
\sum_{y \in [n]}\left[(1+\delta)\widetilde{\pi}(y)+\frac \delta n-\widetilde{P}^t(x,y)\right]= 2\delta+\widetilde{q}(x),
\end{split}
\end{equation}
It is now sufficient to provide an upper bound on $\widetilde{q}(x)$, uniformly over $x\in \veps$. 
This can be derived by bounding above the probability that some  conditions in Definition \ref{def} fail.\\
Condition (i) fails, by definition, with quenched probability $\cQ_{x,t}\left(\frac 1 {n\log^3 n}\right)$, for all $x\in[n]$.\\
Condition (ii) holds with quenched probability $1-o_{\p}(1)$ for all $x \in \veps$, by Proposition~\ref{tree}.\\
Condition (iii) is satisfied with quenched probability bounded below by $\quench_x(X_{s+1} \in \veps)$.
Taking the minimum over $x\in V_\eps$, and thanks to Lemma \ref{lemmaveps}, 
we conclude that (iii) holds with quenched probability  $1-o_{\p}(1)$, uniformly for $x \in \veps$.\\
At last, condition (iv) holds w.h.p. for all $x \in [n]$ due to Lemma~\ref{C}.\\

In conclusion, 
\begin{equation}\label{qto0}
		\max_{x \in \veps}\widetilde{q}(x)\le \max_{x \in \veps}\cQ_{x,t}\left(\frac{1}{n\log^3 n}\right)+ o_{\p}(1)\,.
	\end{equation}
Note that for $\theta= \frac{1}{n\log^3 n}$ and $t= (1+\beta)\tent$
condition \eqref{limitrho} is satisfied with $\rho<1$. Hence, Theorem~\ref{thm_LLN}(i) applies 
to the \rhs in the last display, and ends the proof of \eqref{main2}. 
\qed \\

It now remains to state and prove the result that was applied in order to reduce \eqref{eq2} to \eqref{eq3}.
Set $\beta=3\gamma=\frac {3\eps} {80}$. In the notation introduced above, it holds the following.
\begin{prop} \label{drop-positive-part} 
	
Let $t = s + h + 1$ with $s = (1 - \gamma)\tent$, $\gamma > 0$ as in Definition \ref{def} and $h \equiv h_\eps$ as in \eqref{h_eps}.
Then, for all $\delta>0$, 
\begin{equation}
\p\left(\max_{x \in \veps}\widetilde{P}^t(x,y)
\le (1+\delta)\widetilde{\pi}(y)+\frac{\delta}{n},\forall y \in [n]\right)=1-o(1).
\end{equation}
\end{prop}
\begin{proof}
To prove the statement, we will perform a time-gluing procedure among the first $s$ steps of the walk 
(which is confined \whp in the tree $\mathcal{T}_x(s)$) and the last $h$ steps (where the path to a target end point $y$ is unique). 
Thanks to a partial conditioning on the starting and ending subpaths (of length resp.~$s$ and $h$), 
we will be able to prove a concentration result for the trajectories of length $t$
which will lead to the desired inequality. 

Let us stress that the entire strategy closely follows the proof of Proposition 3.6 in \cite{CCPQ}, 
given in the context of directed configuration models, but requires significant adaptations 
to address the directed Chung--Lu framework.
In particular, while the analysis for the directed configuration model 
relies heavily on combinatorial computations, which reflect the nature 
of the model where the in- and out-degrees are deterministic, 
our approach leverages the independence of connections between vertices, along with appropriate concentration inequalities for the in- and out-degrees. 
This shift is particularly evident in the computations beginning with \eqref{grado}.\\

Given $x\neq y\in[n]$, let $\mathcal{F}=\mathcal{F}(x,y)$ denote the partial environment obtained
after the generation of $\mathcal{T}_x(s)$ and $\mathcal{B}^-_y(h)$. Consider $\kappa_x$ and $\kappa_y$ to be the number of matchings needed to generate respectively the two subgraphs. It holds $\kappa_x=|\mathcal{T}_x(s)|-1$ and $\kappa_y\le |\mathcal{B}^-_y(h)|-1$.

Let $V^-_\mathcal{F}$ denote the set of vertices in $\partial \mathcal{B}^-_y(h)$ 
such that there exists a unique path of length $h$ to $y$, and $V^+_\mathcal{F}$  be the set of unmatched vertices at depth $s$ 
in $\mathcal{T}_x(s)$. 
Note that, by construction,
	\begin{equation} \label{V+}
		\sum_{z \in V^+_\mathcal{F}} \mass(\pat_{x,z}) \le 1,
	\end{equation}
	and
	\begin{equation} \label{V-}
		\sum_{v \in V^-_\mathcal{F}} \muin(v)\mass(\pat_{v,y}) 
		\le \muin P^h(y) =\sum_{v \in [n]} \muin(v)P^h(v,y).
	\end{equation}
With this notation, we develop $\widetilde{P}^t(x,y)$, the probability to follow a nice path of length $t$
from $x$ to $y$, as 
	\begin{equation} \label{ptilda}
		\widetilde{P}^t(x,y)= \sum_{z \in V^+_\mathcal{F}} \sum_{v \in V^-_\mathcal{F}} \mass(\pat_{x,z}) 
		\frac{1}{D^{+}_z}\mass(\pat_{v,y})\one_{\{z \to v\}}\one_{\{\pat \textup{is a nice path}\}},
	\end{equation}
where $\pat=\pat_{x,z}\cup(z,v)\cup\pat_{v,y} $, with a little abuse of notation.
Note that, in this representation of  $\widetilde{P}^t(x,y)$, the last indicator highlights the validity of conditions (i) and (iv) of definition \ref{def} of nice paths, since (ii) and (iii) are satisfied by construction.
	
	We want  study the conditional expectation of \eqref{ptilda} on the partial environment $\mathcal{F}$. 
	By linearity, we are reduced to analyse the random variables $\one_{\{z \to v\}}/D_z^+$ 
	for $z \in V^+_\mathcal{F}$, $v \in V^-_\mathcal{F}$.
	Since the Bernoulli variables $ \one_{\{z \to v\}}$ are independent from the partial environment $\mathcal{F}$, it holds 
	\begin{equation}\label{grado}
		\E\left[ \frac{\one_{\{z \to 	v\}}}{D^+_z}\big| \mathcal{F}\right]=
		p_{zv}\E\left[ \frac{1}{{D}^+_z} \, \big| \mathcal{F}, \one_{\{z\to v\}}=1\right],
	\end{equation}
Some of the indicator functions defining the out-degree of $z$ may have been sampled during the generation of the partial environment $\mathcal{F}=\mathcal{F}(x,y)$. However it holds
$$D_z^+
\ge\sum_{\substack{w \in [n]\,:\\\, (z,w)\notin \, \mathcal{F}\,\cup(z,v)}} \one_{\{z \to w\}} \, +\, \one_{\{z\to v\}}
=: Y^v_z + \, \one_{\{z\to v\}}\,.$$
Thus, we can write
	\begin{equation}
		\E\left[ \frac{\one_{\{z \to v\}}}{D^+_z}\big| \mathcal{F}\right] \le p_{zv} \E\left[\frac{1}{Y_z^v+1} \big|\mathcal{F}\right]=
		\frac{p_{zv}}{\E[Y^v_z| \mathcal{F}]}(1+o(1)),
	\end{equation}
where the last equality follows by Remark \ref{inverse-remark}.
For $y \in [n]$, we then define the events $\mathcal{W}_{y}:=\{\kappa_{y} \le n^{\frac 1 2 +\eps}\}$ 
and $\mathcal{W}:=\cap_{y \in [n]}\mathcal{W}_y$. 
Since $\mathcal{W} \supseteq\mathcal{S}_\eps^-$, where $\mathcal{S}_\eps^-$ is defined in \eqref{S-},
by Lemma \ref{size-in-ball} we get
\begin{equation} \label{W}
	\p(\mathcal{W})\ge \p(\mathcal{S}_\eps^-)=1-o(1)\,.
\end{equation} 
On  $\mathcal{W}_y$, the number of Bernoulli variables removed from $D_z^+$ in the definition of $Y_z^v$ 
 is at most $n^{\frac 12+\eps}$. 
 Moreover, thanks to  \eqref{maxp_xy}, the connection parameter is at most 
 $p_{\text{max}}=o(n^{-\frac 1 2 -\frac \eta 7})$.
 We assume since now on that   $\mathcal{F}\in \mathcal{W}_y$.
 Then, if $\eps < \eta/7$, we get that
$	\E[Y_z^v\,|\mathcal{F}]=\E[D_z^+](1+o(1))=
	\w w_z^+ \log n/n(1+o(1)),$
and consequently we may conclude that 
\begin{equation} \label{one-over-D}
	\E\left[ \frac{\one_{\{z \to v\}}}{D^+_z}\, \big| \mathcal{F}\right]\le\frac{w_v^-}{\w}(1+o(1))=\muin(v)(1+o(1)).
\end{equation}
Taking the conditional average in \eqref{Ptilde} and plugging there \eqref{one-over-D}, 
we obtain that 
\begin{equation}
\begin{split}
		\E[\widetilde{P}^t(x,y)\,|\, \mathcal{F}\,]
		&\le 
		\sum_{z \in V^+_\mathcal{F}} \sum_{v \in V^-_\mathcal{F}} \mass(\pat_{x,z}) 
		\E\left[\frac{\one_{\{z \to v\}}}{D^{+}_z}\,\big|\mathcal{F}\right]\mass(\pat_{v,y})\\
		&\le \sum_{z \in V^+_\mathcal{F}}   \sum_{v \in V^-_\mathcal{F}}\mass(\pat_{x,z})  \muin(v)\mass(\pat_{v,y})(1+o(1)) 
		\\
		& \le \sum_{v \in V^-_\mathcal{F}}  \muin(v)\mass(\pat_{v,y})(1+o(1)) \le \muin P^h(y)(1+o(1)),
\end{split}
\end{equation}
where the last lines follows from \eqref{V+} and \eqref{V-}.
	This implies that for every $\delta>0$ and for $n$ large enough, it holds	
	\begin{equation} \label{F}
		\left(1+\frac \delta 2\right)\E \left[\widetilde{P}^t(x,y) \big| \mathcal{F}\right]\le \left(1+\delta\right) \muin P^h(y)=\left(1+\delta\right)  \widetilde{\pi}(y)\,.
	\end{equation}
Let us consider the random variables
	\begin{equation}
		X_z:=\sum_{v \in V^-_\mathcal{F}} \mass(\pat_{x,z})
		\frac{1}{D^{+}_z} \mass(\pat_{v,y}) \one_{\{z \to v\}}\one_{\{\pat \textup{is a nice path}\}}\, , \quad z \in V^+_\mathcal{F}\,,
	\end{equation}
where $\pat=\pat_{x,z}\cup(z,v)\cup\pat_{v,y} $. 
These random variables are independent. Moreover, thanks to condition (i) of Definition \ref{def}, we have $$\mass(\pat_{x,z})\frac{1}{D^{+}_z}\mass(\pat_{v,y}) \le \frac 1{n\log^3n},$$
and thanks to requirement (iv) of Definition \ref{def}, it holds
	\begin{equation}
		|\{v \in V^-_\mathcal{F} : \pat_{x,z}\cup(z,v)\cup\pat_{v,y}\text{is nice}\}|\le C \log n.
	\end{equation}
	Then, $X_z$ is uniformly bounded in $z \in V^+_{\mathcal{F}}$ by the quantity
	\begin{equation}
		\begin{split}
		M=M(n):= 
			\frac {C \log n}{n\log^3n} =  \frac {C}{n\log^2 n}.
		\end{split}
	\end{equation}
For $a>0$ and $M$ as above,	we can apply the Bernstein inequality to the conditional probability measure $\p(\ \cdot \ |\mathcal{F})$, and get
	\begin{equation} \label{bernstein}
		\p\left(\widetilde{P}^t(x,y)-\E \big[\widetilde{P}^t(x,y) \big| \mathcal{F}\big]\ge a \ \big|\mathcal{F}\right)\le \exp\bigg(-\frac{a^2}{2M( \E [\widetilde{P}^t(x,y) | \mathcal{F}]+a)}\bigg)\,.
	\end{equation}

	Reasoning as in \cite[Prop.~14]{BCS19}, we write $r=n\E [\widetilde{P}^t(x,y) | \mathcal{F}]$ and let 
	$a=\frac\delta n(\frac r 2 +1)
	$. Then the \rhs of \eqref{bernstein} turns to
$$		\exp
		\bigg(-\frac
		{\delta^2(r+ 2)^2}
		{4Mn( r(2+ \delta)+ 2\delta)}
		\bigg) 
		\le  \exp{\left(-\frac{c(\delta)C}{Mn}\right)}
		= \exp{\left(-c(\delta)\log^2 n\right)},
$$
	where $c(\delta)>0$, is obtained optimizing over $r\ge 0$. In this notation, we rewrite \eqref{bernstein} as	
\begin{equation}\label{bern}
		\p\left(\widetilde{P}^t(x,y)\ge \Big( 1+\frac \delta 2\Big)
		\E \left[\widetilde{P}^t(x,y) \big| \mathcal{F}\right]
		+\frac \delta n \ \big|\mathcal{F}\right)
		 \le \exp{\left(-c(\delta)\log^2 n\right)}.
\end{equation}
In conclusion, by \eqref{F} and \eqref{bern}, we get that, for all $\mathcal{F}\in\mathcal{W}_y$,
\begin{equation} \label{powerful}
	\p\left(\widetilde{P}^t(x,y)\ge ( 1+ \delta )\widetilde{\pi}(y) 
	+\frac \delta n \ \big|\mathcal{F}\right) =\exp{\left(-c(\delta)\log^2 n\right)}=o(n^{-3}).
\end{equation}
We are almost done.  Reasoning as in \cite[Prop. 3.6]{CCPQ}, for $x \in \veps$ and $y \in [n]$, let 
$$\mathcal{Z}_{x,y}:= \left\{\widetilde{P}^t(x,y) \ge ( 1+ \delta )\widetilde{\pi}(y)+\frac \delta n\right\}.$$ 

	With a little abuse of notation we can write
$$\p(\cup_{x \in \veps,y \in [n]}\mathcal{Z}_{x,y} \cap \mathcal{W}) 
\le n^2 \max_{x \in \veps,y \in [n]}\p(\mathcal{Z}_{x,y} \cap \mathcal{W}_y) 
\le n^2 \max_{x \in \veps,y \in [n]} \max_{\mathcal{F} \in \mathcal{W}_y}\p(\mathcal{Z}_{x,y} | \mathcal{F})\,,$$
where the last probability is precisely the l.h.s.~in \eqref{powerful}. 
Then, having in mind \eqref{W}, 
\begin{equation}
	\p\big(\cap_{x \in \veps,y \in [n]} \mathcal{Z}_{x,y}\big) 
	\ge 1- \p(\cup_{x \in \veps,y \in [n]}\mathcal{Z}_{x,y} \cap \mathcal{W})-\p(\mathcal{W}^c)=1-o(1)\,,
\end{equation}
which concludes the proof.
\end{proof}

\begin{remark}
	This proof works well also for times $t_\lambda$, lying in the critical window of Theorem \ref{thm_window}, for which $s=(1-4\gamma)\tent(1+o(1))$, as explained in \ref{rem-fine}.
\end{remark}

\subsection{Proof of the lower bound on the mixing time (Eq. \eqref{main1} of Theorem \ref{thm_main})}
 One possible approach to the lower bound consists in achieving inequality \eqref{lower-bound},
 and then applying the law of large number stated in Theorem~\ref{thm_LLN}(ii). 
The bored reader can skip to Subsection \ref{cutoff-window} for this approach. 
 We present here an alternative proof, in the spirit of \cite{CCPQ}, which exploits the equivalent construction of the annealed random walk 
described in \ref{annealed random walk}. \\

The idea of the proof is that, on one hand, the stationary distribution $\pi$ 
is \whp well distributed on $[n]$, in a sense that is specified by Lemma \ref{pi} below
(see also the stronger result stated in the Proposition \ref{prop_piquadro}).
On the other hand, after $t=(1-\beta)\tent$ steps, the random walk concentrates on a set of size at most 
$n^{1-\beta^2}$ which cannot cover the entire graph, and hence the mixing is far to be achieved at this timescale.\\

Formally, for $\beta \in (0,1)$, let  $t=(1-\beta)\tent$ and let $\mathcal{P}_{x,y}^\beta$ denote the set of paths from $x$ to $y$ of lenght $t$ and with probability mass bigger or equal than $1/n^{1-\beta^2}$. 
An easy check shows that 
\begin{equation}\label{UP}
		\sum_{y \in [n]}|\mathcal{P}_{x,y}^\beta|\le n^{1-\beta^2}\,,
	\end{equation}
	and hence the set $S_x:=\{y \in [n]: \mathcal{P}_{x,y}^\beta \neq \emptyset\}$ satisfies $|S_x|\le n^{1-\beta^2}$. 
From the notation of distance in total variation, we can write
\begin{equation}
	\min_{x \in [n]}\|P^t(x,\cdot)-\pi\|_{\text{TV}}
		\ge \min_{x \in [n]} \left(P^t(x,S_x)-\pi(S_x)\right)
		\ge \min_{x \in [n]}	P^t(x,S_x)-\max_{x \in [n]}\pi(S_x).
	\end{equation}
Note that, by definition of $S_x$ and of the quenched probability
$\mathcal Q_{x,t}(\theta)$ in \eqref{quenched-theta}, it holds that
$$P^t(x,S_x)\ge \mathcal Q_{x,t}(n^{1-\beta^2})\,,\qquad \forall x\in[n]\,.$$
We can then apply Theorem~\ref{thm_LLN}(ii) with $\theta=n^{1-\beta^2}$ and $t=(1-\beta)\tent$,
so that the condition 
$\rho=\lim_{n\to\infty} -\frac{\log\theta}{\entropy t}=1+\beta>1$
is satisfied, and conclude that
$$\min_{x \in [n]}	P^t(x,S_x)\ge 1-o_{\mathbb P}(1)\,.$$
Going back to \eqref{UP}, it now remains to show that $\max_{x \in [n]}\pi(S_x)$ is negligible. 
We stress that, by monotonicity of the total variation distance, we may assume $\beta^2<\eta$, where $\eta$ is such that \eqref{assumption} holds. Then, we can  apply the following lemma with 
$\delta:={\beta^2}/{6}$, which provides the desired estimate and ends the proof of the lower bound. \qed
	
\begin{lemma} \label{pi} For all $\delta \in (0,\frac \eta 6)$, with $\eta \in (0,1)$ as in \eqref{assumption}, it holds
	\begin{equation}
		\p(\forall S \subset [n] \text{ such that } |S|\le n^{1-6 \delta}\,:\,\pi(S)\le n^{-\delta/2})=1-o(1).
	\end{equation}
\end{lemma}
\begin{proof}
Let us first define, for any $y \in [n]$ and $t' \in \mathbb{N}$,
 	\begin{equation}\label{law_t}
		\mu_{t'}(y):=\frac 1 n \sum_{x \in [n]}P^{t'}(x,y)\,.
	\end{equation}
By the properties of the total variation distance (see \cite[4.4]{LevPer:AMS2017}), it holds that
	\begin{equation} 
		\|P^{ks}(x,\cdot)-\pi\|_{\textup{TV}} \le (2\|P^s(x,\cdot)-\pi\|_{\textup{TV}})^k\,,
	\end{equation}
for any $k,s\in\mathbb{N}$. Thanks to the upper bound \eqref{main2},
 it holds \mbox{$\|P^{2\tent}(x,\cdot)-\pi\|_{\textup{TV}} \le 1/2e$}.
Then, choosing $k=\log^2 n$, $s=2\tent$, and setting $T=ks$, we get
	\begin{equation} 
	\|P^{T}(x,\cdot)-\pi\|_{\textup{TV}} \le (2 \|P^{2\tent}(x,\cdot)-\pi\|_{\textup{TV}})^{\log^{{3}/{2}} n}\le e^{-\log^{{3}/{2}} n}\,,
\end{equation}
which implies
	 \begin{equation}\label{pi-mu}
	 	\max_{v \in[n]}|\pi(v)-\mu_T(v)|=o(e^{-\log^{3/2}(n)})\,,
	 \end{equation}
As a consequence, we can prove the thesis for $\mu_T$ in place of $\pi$.
		
To prove the statement, it is now sufficient to show that, given $L:=\lceil n^{1-6\delta}\rceil$, then
	\begin{equation}
		\max_{S:|S|=L}\p(\mu_T(S)\ge n^{-\delta})=o(n^{-L})\,.
	\end{equation}

So let $S\subset [n]$ with $|S|=L$, set $K=\delta^{-1}L$, and consider the annealed law $\p_{\textup{unif}}^{\textup{an},K}$ of the process $(X^{(k)})_{k \in \{1,\dots,K\}}$ 
defined as in Subsection \ref{annealed random walk}, for $T=\log^2 n \cdot \tent$. For every $j\le K$, let $B_j$ be the event defined by the following property: the first $j$ walks end in $S$. It holds
	\begin{equation}
		\E[\mu_T(S)^K]=\p_{\textup{unif}}^{\textup{an},K}(B_K) 
		= \p_{\textup{unif}}^{\textup{an},K}(B_1)\prod_{j=2}^{K} \p_{\textup{unif}}^{\textup{an},K}(B_j  |  B_{j-1})\,.
	\end{equation}
Since $6\delta < \eta$,  we can apply the same argument
used in \eqref{hoelder}, with $p=2+6\delta$, and get 
\begin{equation}
	\sum_{v \in S} w_v^- = O(n^{\frac 12 -\delta}L^{\frac 12 + \delta})=O(n^{1 -3{\delta} - 6\delta^2})\,,
\end{equation}
where we used that $\frac{1}{2+6\delta} < \frac 1 2 - \delta$ and that 
$|S|=L=\lceil n^{1-6\delta}\rceil$.

Given $B_{j-1}$, the $j$-th trajectory can end in $S$ if it replicates from the beginning one of the previous $j-1$ trajectories (this happens with probability at most $\frac {KT} n$), or if it enters at least once the set $S$ or the set formed by the $j-1$ trajectories. Non-fresh vertices (i.e.,~the ones belonging to the previous trajectories) affect only logarithmically  the order of $L$,
 and the probability of entering $S$ from a fresh vertex at a given step is bounded by
	\begin{equation}
		\max_{x \in [n]}\sum_{v \in S}p_{x,v} \le M_1 \frac{\log n}n\sum_{v \in S} w_v^- = o(n^{-{3\delta}}).
	\end{equation}
	Since the $j$-th trajectory has $T=O(\log^3n)$ steps, we conclude that
	\begin{equation}
		\p_{\textup{unif}}^{\textup{an},K}(B_j | B_{j-1}) \le \frac{KT}{n} + T\, o\big(n^{-3\delta}\big) = o(n^{-2\delta}).
	\end{equation}	
This proves that $\E[\mu_T(S)^K]=o(n^{-2\delta K})=o(n^{-2L})$.
By Markov's inequality, and being $K=\delta^{-1}L$, we obtain
\begin{equation}
\p(\mu_T(S)\ge n^{-\delta}) \le \frac{\E[\mu_T(S)^K]}{n^{-L}}=o(n^{-L})\,,
\end{equation}
and conclude with a union bound on the $O(n^{L})$ sets $S$ with $|S|=L$.
\end{proof}

\subsection{Proof of Theorem \ref{thm_window}: Cutoff window} \label{cutoff-window}

We are going to provide upper and lower bounds on the total variation distance
which appears in the statement.

\textbf{We first prove the upper bound.}\\
Recall the notation introduced in the Theorems \ref{thm_LLN}, \ref{CLT} and in Eq.~\eqref{q-secondo}, 
and take a reference time 
$\tcrit:=\tent+\lambda \textbf{\textup{w}}_n+o(\textbf{\textup{w}}_n)$ with  
$\lambda \in \mathbb{R}$ fixed. 
Since $\var(S_{t_{\lambda}})=\sigma^2 t_{\lambda}$, choosing $\theta=\frac 1 n$, it holds  that
	\begin{equation}
		\begin{split}
		\frac{\entropy t_{\lambda} + \log \theta}{\sqrt{\var(S_{t_{\lambda}})}} = 
		&
		\frac
		{\lambda \sigma\sqrt{\frac{\log n}{\log \log n}} (1+o(1))}
		{\sigma\sqrt{\frac{\log n}{\log \log n}(1+o(1))}}
				\xrightarrow[n \to +\infty]{} \lambda\, , 
		\end{split}
	\end{equation}
and we are then under the hypothesis \eqref{lambda} of Theorem~\ref{CLT}.
Thanks to this result, together with the inequality \eqref{qto0}, we get that
 for every $\delta>0$ and \whp
	\begin{equation}
		\max_{x \in \veps}\widetilde{q}(x) \le \int_{\lambda}^{+\infty} \frac{1}{\sqrt{2\pi}}e^{-\frac{u^2}{2}} \,du +\delta\,.
	\end{equation}
Applying Proposition \ref{drop-positive-part}, we may conclude that for every $\delta>0$ 
	and \whp
	\begin{equation}
			\|P^{t_{\lambda}}(x,\cdot)-\widetilde{\pi}\|_{\textup{TV}} \le 2\delta+\widetilde{q}(x) 
			\le \int_{\lambda}^{+\infty} \frac{1}{\sqrt{2\pi}}e^{-\frac{u^2}{2}} \,du+3\delta \,.
	\end{equation}
	
\textbf{For the lower bound}, we first observe that, for $\theta \in (0,1)$, 
	\begin{equation}
		P^{t_{\lambda}}(x,y) \ge \sum_{\pat \in \mathcal{P}(x,y,t_{\lambda},G)} \mass(\pat)\one_{\mass(\pat)\le \theta}\,.
	\end{equation}
	Then, for every distribution $\nu$ on $[n]$,
	\begin{equation}
		\nu(y)-\sum_{\pat \in \mathcal{P}(x,y,t_{\lambda},G)} \mass(\pat)\one_{\mass(\pat)\le \theta}\le \left[\nu(y)-P^{t_{\lambda}}(x,y)\right]^+ + \nu(y)\one_{P^{t_{\lambda}}(x,y)>\theta}.
	\end{equation}
	Summing over $y \in [n]$, using that there are less than $1/\theta$ vertices 
	such that $P^{t_{\lambda}}(x,y)>\theta$, and by the Cauchy-Schwarz inequality, we get
	\begin{equation} \label{ineq}
		\cQ_{x,t_{\lambda}}(\theta)\le 
		\|\nu-P^{t_{\lambda}}(x,\cdot)\|_{\textup{TV}}+ 
		\sqrt{\frac{1}{\theta}\sum_{y \in [n]}\nu^2(y)}\,.
	\end{equation}
	We just need to show that for suitable choices of $\nu$ and $\theta \in [0,1]$, 
	\eqref{ineq} implies the claimed statement.
\begin{enumerate}
\item	
A quite straightforward proof of this fact can be done under a further assumption
on the  weigths $(w_x^-)_{x\in[n]}$. 
Explicitly, let $w_{\text{max}}^-(n):= \max_{x \in [n]}w_x^-$, and assume that
\begin{equation}\label{extra}
w_{\text{max}}^-(n)= o(e^{\sqrt{\log n}})\,.
\end{equation}
Choosing $\nu=\widetilde{\pi}$ and $\theta=w_{\text{max}}^-(n)\frac{ \log^4 n}{n}$, 
we want to show that
	\begin{equation} \label{Epi}
		\frac 1 \theta \E\left[\sum_{x \in [n]}\widetilde{\pi}^2(y)\right]=o\left(1\right)\,.
	\end{equation}
Then Markov's inequality will be sufficient to conclude that 
\begin{equation}\label{ultima}
			\cQ_{x,t_{\lambda}}(\theta)\le 
		\|\nu-P^{t_{\lambda}}(x,\cdot)\|_{\textup{TV}}+ o_{\p}(1)\,,
\end{equation} 
and the desired lower bound will be a consequence of the central limit Theorem~\ref{CLT}.

To prove \eqref{Epi}, first note that, since $\widetilde{\pi}=\muin P^h$,  
$$
\E\left[\sum_{x \in [n]}\widetilde{\pi}^2(y)\right]= \p^{\textup{unif}}_{\muin}(X^{(1)}_h=X^{(2)}_h)\,,
$$
where  $X^{(1)}$ and $X^{(2)}$ are two random walks as defined in Subsection \ref{annealed random walk},
 and with initial distribution $\muin$.
Thanks to assumptions \eqref{maxa-} and \eqref{extra}, the probability that they start 
from the same vertex is less than 
	\begin{equation*}
		\mumax= o(w_{\text{max}}^-(n)/n)=o(1),
	\end{equation*}
	On the other hand, the probability that $X^{(2)}$ meets $ X^{(1)}$ at a certain step 
	$0<s\le h$ is less than $(h+1)^2 p_{\text{max}}$, where $p_{\text{max}}=o(w_{\text{max}}^-(n)\log n/n)$. 
Thanks to the assumption \eqref{extra}, we globally get  
	\begin{equation}
			\frac 1 \theta \E\left[\sum_{x \in [n]}\widetilde{\pi}^2(x)\right]
			=\frac 1 \theta \p^{\textup{unif}}_{\muin}(X^{(1)}_h=X^{(2)}_h)
			=O\left(\frac 1 {\log n}\right),
	\end{equation}
and thus conclude the proof of \eqref{Epi} and of the lower bound.

\item To get rid of assumption \eqref{extra}, we may proceed in a similar way but 
choosing $\nu = \pi$ and $\theta = \frac{1}{n}\log^{8} n$, so to stay again under the hypothesis 
\eqref{lambda} of Theorem \ref{CLT}.
To recover the analogue of \eqref{Epi}, with $\pi$ instead of $\widetilde{\pi}$, we will
need to apply the Proposition \ref{prop_piquadro} below.
Given this, we can easily recover the estimate \eqref{Epi}, 
and then conclude with the application of Theorem~\ref{CLT}. \qed
\end{enumerate}
\begin{prop} \label{prop_piquadro} 
In the above notation and setting, it holds that
	\begin{equation}
		\E\left[\sum_{x \in [n]}\pi^2(x)\right]	\le C_1\frac{\log^{6}n}n\,,
	\end{equation}
where $C_1>0$ is the finite constant given in Lemma \ref{lemma-selfintersection}. 
\end{prop}

\begin{proof}
	Let $T=\log^2 n$.
	Let $X^{(1)}$ and $X^{(2)}$ be two independent random walks of length $T$, moving on the same random graph
and with initial distribution $\textup{Unif}([n])$.
Note that, according to Subsection \eqref{annealed random walk}, their joint annealed law is equivalently described
by the measure $\p_{\textup{unif}}^{\textup{an,2}}$, that for the sake of readability we simply write $\punifan$.

Denoting by $\mu_{T}$ their common distribution at time $T$, as in \eqref{law_t}, 
we can immediately argue from \eqref{pi-mu} that being $T\gg \tent$, then
	\begin{equation}
		\E\left[\sum_{x \in [n]}(\pi(x))^2\right]
		=\E\left[\sum_{x \in [n]}(\mu_T(x))^2\right](1+o(1))
		=\punifan(X^{(1)}_T=X^{(2)}_T)(1+o(1))\,.
	\end{equation}
We then focus on the probability on the \rhs of the above display, that will be estimated
using similar ideas to those appeared in Lemmas \ref{lemma-selfintersection} and \ref{lemmaiid}.

At first, let $\mathcal{T}$ denote the first time such that the trajectory of $X^{(2)}$ meets 
that of $X^{(1)}$, 
formally given by
$\mathcal{T}:=\min\{s >0\, :\, \exists u\le T \mbox{ such that } X_s^{(2)}=X_u^{(1)} \}\,,$
so that
	\begin{equation}\label{fine}
		\begin{split}
			\punifan(X^{(1)}_T=X^{(2)}_T) 
		\le \punifan(\mathcal{T} \le T) 
			 = \sum_{t=0}^T \punifan(\mathcal{T} =t ). 
		\end{split}
	\end{equation}
Since the initial measure is uniform over $[n]$, we immediately get that
$\punifan (\mathcal{T}=0)\le \frac{T}{n}$.

For $t=1,\dots, T$, it is instead convenient to consider the events 
$$\mathcal{A}_s^{j}\equiv \mathcal{A}_s^{X^{(j)}}\,,\,\mbox{ for }j=1,2 \mbox{ and } s\in\{0,\ldots T\}\,,$$
given in \eqref{event_A}, and to introduce,  for any $s,t\in\{0,\ldots T\}$,
 the events
$$\mathcal{B}_{s,t}^{1,2}:=\big\{X^{(2)}_{v}\neq X^{(1)}_u ,\, \forall  u\in\{0,\ldots, s-1\} 
\mbox{ and }v\in\{s,\ldots, t-1\}\big\}\,,
$$
which are analogues of the events defined in \eqref{event_B}.
With this notation, we can first write
\begin{equation}\label{tau}
	\punifan(\mathcal{T}=t)\le \punifan(\mathcal{T} = t,
\mathcal{A}_t^2)+\punifan((\mathcal{A}_t^2)^c)\,,
	\end{equation}
	where $\punifan((\mathcal{A}_t^2)^c)\le  \punifan((\mathcal{A}_T^2)^c)\le  C_1\log^{4}n/n$ 
	due to Lemma \eqref{lemma-selfintersection}, and then express the first summand as
	\begin{equation}\label{totale}
		\begin{split}
			\punifan(\mathcal{T} = t,\mathcal{A}_t^2) 
			&= 
			\sum_{s=0}^T \sum_{z \in [n]}	\punifan(X^{(2)}_t= X^{(1)}_s=z,\mathcal{B}_{T,t}^{1,2}\cap\mathcal{A}_t^2)\,.
			\end{split}
	\end{equation}

Conditioning over the whole trajectory of $X^{(1)}$, we get
	\begin{equation}\label{mela}
		\begin{split}
			&\punifan(X^{(2)}_t= X^{(1)}_s=z,\mathcal{B}_{T,t}^{1,2}\cap\mathcal{A}_t^2)\\
			&=\sum_{\substack{v \in [n]^{T+1}: \\ v_s=z}} \punifan\left(X^{(2)}_t= z,\, \mathcal{A}_{t}^2
			|(X^{(1)}_s)_{s\le T}=v, \mathcal{B}_{T,t}^{1,2}\right)
			 \punifan\left((X^{(1)}_s)_{s\le T}=v, \mathcal{B}_{T,t}^{1,2} \right)= \\
			&=\sum_{\substack{v \in [n]^{T+1}: \\ v_s=z}} \widetilde{\p}^{\textup{an}}(X_t= z,\mathcal{A}_t) 
			\punifan\left((X^{(1)}_s)_{s\le T}=v, \mathcal{B}_{T,t}^{1,2} \right)\, ,
		\end{split}
	\end{equation}
	where 	$\widetilde{\p}^{\textup{an}}(\cdot):=\widetilde{\E}[\quench_{\textup{unif}}(\cdot)]$ 
	denotes the annealed law induced by a Chung--Lu probability measure $\widetilde{\p}$ on a graph 
	with vertex-set $[n] \setminus \{v_k\}_{k\in[0,T]\setminus \{s\}}$ and $X$ is a simple random walk 
	with initial uniform distribution. 
	To the sake of readability we do not stress the dependence of $\widetilde{\p}$ on the path $v$.  
	
Thanks to  Lemma \ref{lemmaiid}, $\widetilde{\p}^{\textup{an}}(X_t= z,\mathcal{A}_t)\le \widetilde{\p}^{\textup{an}}(X_t= z,\mathcal{L}_{t-1})=\muin(z)(1+o(1))$
uniformly over the paths \mbox{$v\in [n]^{T+1}$} so that, inserting this value in the last display, 
 we get
\begin{equation}\label{int}
\begin{split}
\punifan(X^{(2)}_t= X^{(1)}_s=z,\mathcal{B}_{T,t}^{1,2}\cap\mathcal{A}_t^2)
&\le \muin(z)\punifan(X^{(1)}_s=z, \mathcal{B}_{T,t}^{1,2})(1+o(1))\,.
\end{split}
\end{equation}
As a further application of Lemmas  \ref{lemmaiid} and \ref{lemma-selfintersection}, it holds
that
$$\punifan(X^{(1)}_s=z, \mathcal{B}_{T,t}^{1,2})\le 
\punifan(X^{(1)}_s=z) \le \punifan(X^{(1)}_s=z,\mathcal{A}^{1}_s) + 
	\punifan((\mathcal{A}^{1}_s)^c)\le\muin(z)(1+o(1))\,,$$
and altogether, going back to Eq.~\eqref{totale} and replacing the value of $T$, we obtain  
	\begin{equation}
		\begin{split}
			\punifan(\mathcal{T} = t,\mathcal{A}_t^2) 
			&\leq 
			\sum_{s=0}^T \sum_{z \in [n]}\muin(z)^2(1+o(1)) 
			 =O \left( \frac{T}{n}\right)\,,
		\end{split}
	\end{equation}
where in the last identity we used the approximation \eqref{square}.
We conclude that the leading term in \eqref{tau} is indeed provided by $\punifan({\mathcal{A}_t^2}^c)$,
so that
		$$\punifan(\mathcal{T}=t)\le  C_1\frac{\log^{4}n}{n} (1+o(1))\,, $$
which inserted  	in \eqref{fine} yields the claimed inequality.
\end{proof}

\end{document}